\documentclass[11pt,a4paper,openany]{amsart}

\usepackage{amsmath,amsfonts,ams}
\usepackage{latexsym}
\usepackage{amssymb}
\usepackage{color}
\usepackage{epsfig}
\usepackage{graphicx}
\usepackage{subfigure}
\usepackage{mathrsfs}
\usepackage{amscd}
\usepackage{amsthm}
\usepackage{color}
\usepackage{cite}
\usepackage{epstopdf}
\usepackage{fancyhdr}

\numberwithin{equation}{section}
\newtheorem{proposition}{Proposition}[section]
\newtheorem{definition}{Definition}[section]
\newtheorem{lemma}{Lemma}[section]
\newtheorem{theorem}{Theorem}[section]
\newtheorem{corollary}{Corollary}[section]
\newtheorem{remark}{Remark}[section]

\let\pa=\partial
\let\al=\alpha

\def\C{\mathop{\mathbb C\kern 0pt}\nolimits}
\def\DD{\mathop{\bf D\kern 0pt}\nolimits}
\def\K{\mathop{\bf K\kern 0pt}\nolimits}
\def\N{\mathop{\bf N\kern 0pt}\nolimits}
\def\Q{\mathop{\bf Q\kern 0pt}\nolimits}

\newcommand{\la}{\lambda}

\newcommand{\beq}{\begin{equation}}
\newcommand{\eeq}{\end{equation}}
\newcommand{\ben}{\begin{eqnarray}}
\newcommand{\een}{\end{eqnarray}}
\newcommand{\beno}{\begin{eqnarray*}}
\newcommand{\eeno}{\end{eqnarray*}}

\makeatletter
\newcommand{\Extend}[5]{\ext@arrow0099{\arrowfill@#1#2#3}{#4}{#5}}
\makeatother

\begin{document}
\title{Scattering Theory  for The Defocusing  Fourth-order Schr\"odinger  Equation}

\author{Changxing Miao}
\address{Institute of Applied Physics and Computational Mathematics,
P. O. Box 8009,\ Beijing,\ China,\ 100088;} \email{miao\_changxing@iapcm.ac.cn}

\author{Jiqiang Zheng}
\address{The Graduate School of China Academy of Engineering Physics, P. O. Box 2101, Beijing, China, 100088}
\email{zhengjiqiang@gmail.com} \maketitle
\begin{abstract}
In this paper, we study the global well-posedness and scattering
theory for the defocusing  fourth-order nonlinear Schr\"odinger
equation (FNLS) $iu_t+\Delta^2 u+|u|^pu=0$  in dimension $d\geq9$.
We prove that if the solution $u$ is apriorily bounded in the
critical Sobolev space, that is, $u\in L_t^\infty(I;\dot
H^{s_c}_x(\R^d))$ with all $s_c:=\frac{d}2-\frac4p\geq1$ if $p$ is
an even integer or $s_c\in[1,2+p)$ otherwise, then $u$ is global and
scatters. The impetus to consider this problem stems from a series
of recent works for the energy-supercritical and energy-subcritical
nonlinear Schr\"odinger equation (NLS) and nonlinear wave equation
(NLW). We will give a uniform way to treat the energy-subcritical,
energy-critical and energy-supercritical FNLS, where we utilize the
strategy derived from concentration compactness ideas to show that
the proof of the global well-posedness and scattering is reduced to
exclude the existence of three scenarios: finite time blowup;
soliton-like solution and low to high frequency cascade. Making use
of the No-waste Duhamel formula, we deduce that the energy or mass
of the finite time blow-up solution is zero and so get a
contradiction. Finally, we adopt the double Duhamel trick, the
interaction Morawetz estimate and interpolation to kill the last two
scenarios.
\end{abstract}

\begin{center}
 \begin{minipage}{120mm}
 \end{minipage}
 \end{center}
\section{Introduction}
\setcounter{section}{1}\setcounter{equation}{0} This paper is mainly
concerned with the Cauchy problem of the defocusing fourth-order
Schr\"odinger equation (FNLS)
\begin{align} \label{equ1.1}
\begin{cases}    iu_t+\Delta^2u+f(u)=  0,  \qquad  (t,x) \in
\mathbb{R}\times\mathbb{R}^d,~d\geq9,\\
u(0,x)=u_0(x)\in \dot H^{s_c}(\R^d),\end{cases}
\end{align}
where $f(u)=|u|^pu,~u$ is a complex-valued function defined in
$\mathbb{R}^{1+d}$,  $\Delta$ is the Laplacian in $\mathbb{R}^{d},$
and $s_c:=\frac{d}2-\frac4p.$

If the solution $u$ of \eqref{equ1.1} has sufficient decay at
infinity and smoothness, it conserves  mass
\begin{equation}
M(u)=\int_{\mathbb{R}^d}|u(t,x)|^2dx=M(u_0)
\end{equation}
and energy
\begin{equation}
E(u)=\frac12\int_{\mathbb{R}^d}|\Delta
u|^2dx+\frac1{p+2}\int_{\mathbb{R}^d}|u(t,x)|^{p+2}dx=E(u_0).
\end{equation}
As similarly explained in \cite{GiV00},  the above quantities are
also conserved for the energy solutions $u \in C^0_t(\mathbb{R},
H^2(\mathbb{R}^d))$. We call $\dot{H}^2_x(\R^d)$ the energy space.

The equation \eqref{equ1.1} has the scaling invariance symmetry:
\begin{equation}\label{scale}
u(t,x)\mapsto \la^{\frac4p}\ u(\la^4\ t,\la\ x), ~\forall~\la>0
\end{equation}
in the sense that both the equation and the $\dot{H}^{s_c}$-norm are
invariant under the scaling transformation:
$$\|u_\lambda\|_{\dot{H}^{s_c}(\R^d)}=\|u\|_{\dot{H}^{s_c}(\R^d)}.$$

We call FNLS \eqref{equ1.1} the energy-subcritical when
$p<\frac8{d-4}$, which corresponds to $s_c<2$, in particular, it is
called the mass-critical when $p=\frac8d$, corresponding to $s_c=0$;
\eqref{equ1.1} refers to energy-critical when $d\geq5$ and
$p=\frac8{d-4}$, corresponding to $s_c=2$; and \eqref{equ1.1} refers
to energy-supercritical when $p>\frac8{d-4}$, corresponding to
$s_c>2$.

Fourth-order Schr\"odinger equations have been introduced by Karpman
\cite{Karpman} and Karpman and Shagalov \cite{KarpmanShagalov} to
take into account the role of small fourth-order dispersion terms in
the propagation of intense laser beams in a bulk medium with Kerr
nonlinearity. Such fourth-order Schr\"odinger equations are written
as
\begin{equation}
i\pa_tu+\Delta^2u+\varepsilon\Delta u+f(|u|^2)u=0,
\end{equation}
where $\varepsilon\in\{\pm1,0\}.$ Such equations have been studied
from the mathematical viewpoint in Fibich, Ilan and Papanicolaou
\cite{Fibich} who describe various properties of the equaion in the
subcritical regime, with part of their analysis relying on very
interesting numerical developments. Related reference is
\cite{Ben-Artzi} by Ben-Artzi, Koch, and Saut, which gives sharp
dispersive estimates for the biharmonic Schr\"odinger operator which
lead to  the Strichartz estimates for the fourth-order
Schr\"{o}dinger equation, see also
\cite{MiaoZhang,Pausader,Pausader1}. Guo and Wang \cite{GuoWang} who
prove global well-posedness and scattering in $H^s$ for small data.
For other special fourth order nonlinear Schr\"{o}dinger equation,
please refer to \cite{Segata,ZZ,Z}. For FNLS \eqref{equ1.1}, the
defocusing energy-critical case with nonlinearity given by $f(u)
=|u|^\frac8{d-4}u$ was handled by Pausader \cite{Pausader,Pausader1}
in dimension $d=8$, in which case the nonlinearity is cubic, and
Miao, Xu and Zhao \cite{MiaoXuZhao} in dimension $d\geq9$. We also
refer to Miao, Xu and Zhao \cite{MiaoXu} and Pausader \cite{Pau} for
the focusing case with radially symmetrical initial data. For the
defocusing mass-critical case with nonlinearity given by
$f(u)=|u|^\frac8du$, we refer to Pausader and Shao \cite{PB}, Xia
and Pausader \cite{XP}.

On the other hand, the global well-posedness and scattering theory
for the nonlinear Schr\"odinger equations (NLS)
\begin{equation}
i\pa_tu-\Delta u\pm|u|^pu=0,\quad (t,x)\in\R\times\R^d
\end{equation}
 have been intensively studied recently, most notably
by Bourgain \cite{Bo99a}, Colliander, Keel, Staffilanni, Takaoka and
Tao \cite{CKSTT07}, Kenig and Merle \cite{KM} and Killip and Visan
\cite{KV20101} and Visan \cite{Visan2007,Visan2011} for the
energy-critical case and Tao, Visan and Zhang \cite{TVZ2007},
Killip, Tao and Visan \cite{KTV2009}, Killip, Visan and Zhang
\cite{KVZ2008} and Dodson \cite{Dodson3,Dodson2,Dodson1,Dodson} for
the mass-critical case.

So far, there is no technology for treating large-data NLS without
some a priori control of a critical norm other than the
energy-critical NLS and mass-critical NLS. In \cite{KM2010},
Kenig-Merle first showed that if the radial solution $u$ to NLS
obeys $u\in L_t^\infty(I; \dot{H}^{s_c}(\R^3))$ with $s_c=\frac12$,
then $u$ is global and scatters, where they utilized their
concentration compactness technique as in \cite{KM}, together with
the Lin-Strauss Morawetz inequality which scales like
$\dot{H}^\frac12_x(\R^d)$ and is scaling-critical in this case.
Thereafter, Killip--Visan \cite{KV2010} proved  such result for NLS
in some energy-supercritical regime. In particular, they deal with
the case of a cubic nonlinearity for $d\geq5$, along with some other
cases for which $s_c>1$ and $d\geq5$, where the restriction to high
dimensions comes from the double Duhamel trick. Recently, Murphy
\cite{Mu} considers the energy-subcritical NLS by making use of the
tool ``long time Strichartz estimate" developed by Dodson
\cite{Dodson3} for almost periodic solutions in the mass-critical
setting.

In this paper, we will give a uniform way to treat the
energy-subcritical, energy-critical and energy-supercritical FNLS in
dimension $d\geq9$. We remark that the arguments in this paper also
work for the energy-critical and some energy-subcritical NLS in
dimension $d\geq5$.

\vskip 0.2in Now we introduce some background materials.
\begin{definition}[solution]\label{def1.1}
 A function $u:~I\times\R^d\to\mathbb{C}$ on a nonempty
time interval $I\subset\R$ is a strong solution to \eqref{equ1.1} if
$u\in C_t(K; \dot{H}^{s_c}_x(\R^d))\cap
L_{t,x}^{\frac{d+4}4p}(K\times\R^d)$ for any compact interval
$K\subset I$ and for any $t,t_0\in I,$ it  obeys the Duhamel
formula:
\begin{equation}\label{duhamel.1}
u(t,x)=e^{i(t-t_0)\Delta^2}u(t_0)-i\int_{t_0}^te^{i(t-s)\Delta^2}f(u(s))ds.
\end{equation}
 We say that the interval $I$ is
the lifespan of $u$. We call $u$ a maximal-lifespan solution if the
solution cannot be extended to any strictly larger interval. In
particular, if $I=\R,$ then we say that $u$ is a global solution.
\end{definition}

The solution lies in the space $
L_{t,x}^{\frac{(d+4)p}4}(I\times\R^d)$ locally in time is natural
since by Strichartz estimate (see Proposition \ref{prop1} below),
the linear flow always lies in this space.  Also, if a solution $u$
to \eqref{equ1.1} is global, with $\|u\|_{
L_{t,x}^{\frac{d+4}4p}(\R\times\R^d)}<+\infty$, then it scatters in
both time directions in the sense that
 there exist unique $v_\pm\in\dot{H}^{s_c}_x(\R^d)$ such that
\begin{equation}\label{1.2.1}
\big\|u(t)-e^{it\Delta^2}v_\pm\big\|_{\dot{H}^{s_c}_x(\R^d)}
\longrightarrow 0,\quad \text{as}\quad t\longrightarrow \pm\infty.
\end{equation}
In view of this, we define
\begin{equation}\label{scattersize1.1}
 S_I(u)=\|u\|_{
L_{t,x}^{\frac{d+4}4p}(I\times\R^d)}
\end{equation}
as the scattering size of $u$.

Closely associated with the notion of scattering is the notion of
blow-up:
\begin{definition}[Blow-up]\label{def1.2.1} Let $u:I\times\R^d\to\mathbb{C}$ be a  maximal-lifespan solution to \eqref{equ1.1}. If there
exists a time $t_0\in I$ such that $S_{[t_0,\sup I)}(u)=+\infty$,
then we say that the solution $u$
 blows up forward in time. Similarly, if there exists a time $t_0\in
 I$ such that $S_{(\text{inf}~
I,t_0]}(u)=+\infty$, then we say that  $u(t,x)$ blows up backward in
time.
\end{definition}

Now we state our main result.

\begin{theorem}\label{theorem}
Assume that $d\geq9,$ and $s_c\geq1$ if $p$ is an even integer or
$s_c\in[1,2+p)$ otherwise.  Let $u:~I\times\R^d\to\C$ be a
maximal-lifespan solution to \eqref{equ1.1} such that
\begin{equation}\label{assume1.1}
\|u\|_{L_t^\infty(I; \dot{H}^{s_c}_x(\R^d))}<+\infty.
\end{equation} Then $I=\R$, and the solution $u$  scatters in the sense \eqref{1.2.1}.
\end{theorem}

\begin{remark}
$(i)$ We remark that the balance between the bounds provided by
Lemma \ref{adecay1} and the bound required by Theorem \ref{regular}
by making use of the double Duhamel formula is the source of our
constraint to dimensions $d\geq9.$ More precisely, as we will see in
the below, \eqref{lregd21} provides the $L_t^\infty L_x^q$ bounds
for $q\geq2p,$ while \eqref{zhib} requires this bound with
$q<\frac{pd}4$. These conditions on $q$ impose the restriction
$d\geq9.$

$(ii)$ Our restriction $s_c\geq1$ serves to simplify the analysis
for the local theory, which still becomes a bit complicated.
However, modifying the argument in the local theory, one may extend
Theorem \ref{theorem} to $s_c\geq\frac12$  which  enables us to
adopt the interaction Morawetz inequality $($see Lemma
\ref{morawetz1} below$)$.

$(iii)$ Finally, we also need that the nonlinearity obeys a certain
smoothness condition; more precisely, we ask that $s_c<2+p$ when $p$
is not an even integer. The role of this restriction is to allow us
to take $(s_c-1)$-many derivatives of the nonlinearity $f(u)$. This
is in sharp contrast with NLS, where the restriction for the
regularity $s_c<1+p$ when $p$ is not an even integer. The main
reason is the Strichartz estimate since there is the smoothing
effect for all higher-order nonlinear Schr\"odinger equations, see
Proposition 2 in \cite{MZ}. This enables us to consider $s_c<2+p$
for $p$ being not even integer in FNLS.
\end{remark}

\subsection{The outline of the proof of Theorem \ref{theorem}}

 For each $E>0$, let us define
$\Lambda(E)$ to be the quantity
$$\Lambda(E):=\sup\Big\{S_I(u):~u:~I\times\R^d\to \C\ \text{such\ that}\
\sup_{t\in I}\big\|u\big\|_{\dot H^{s_c}_x(\R^d)}\leq E\Big\},$$
where $u$ ranges over all solutions to \eqref{equ1.1} on the
spacetime slab $I\times\R^d$ with $\big\|u\big\|_{\dot
H^{s_c}(\R^d)}\leq E.$ Thus, $\Lambda:\ [0,+\infty)\to [0,+\infty)$
is a non-decreasing function. Furthermore, from  the small data
theory, see Proposition \ref{lwph}, one has
$$\Lambda(E)\lesssim E\quad \text{for}\quad E\leq\eta_0,$$
where $\eta_0=\eta(d)$ is the threshold from the small data theory.

From the stability theory (see Corollary \ref{long1} below), we know
that $\Lambda$ is continuous. Thus, there is a unique critical
$E_c\in(0,+\infty]$ such that $\Lambda(E)<+\infty$ for $E<E_c$ and
$\Lambda(E)=+\infty$ for $E\geq E_c$. In particular, if
$u:~I\times\R^d\to \C$ is a maximal-lifespan solution to
\eqref{equ1.1} satisfying $\sup\limits_{t\in I}\big\|u\big\|_{\dot
H^{s_c}_x(\R^d)}< E_c,$ then $u$ is global and moreover,
$$S_\R(u)\leq L\big(\big\|u\big\|_{L_t^\infty(\R;\dot H^{s_c}(\R^d))}\big).$$ Therefore, the proof of Theorem \ref{theorem} is
equivalent to show $E_c=+\infty.$ We argue by contradiction. The
failure of Theorem \ref{theorem} would imply the existence of very
special class of solutions; that is the almost periodicity modulo
symmetries:

\begin{definition}\label{apms}
Let $s_c\geq1.$ A solution  $u$ to \eqref{equ1.1} with
maximal-lifespan $I$ is called almost periodic modulo symmetries if
$u$ is bounded in $\dot{H}_x^{s_c}(\R^d)$ and there exist functions
$N(t):~I\to\R^+,~x(t):~I\to\R^d$ and $C(\eta):\R^+\to\R^+$ such that
for all $t\in I$ and $\eta>0$,
\begin{equation}\label{apss}
\int_{|x-x(t)|\geq\frac{C(\eta)}{N(t)}}\big||\nabla|^{s_c}
u(t,x)\big|^2dx\leq\eta
\end{equation}
and
\begin{equation}\label{apsf}
\int_{|\xi|\geq C(\eta)N(t)}|\xi|^{2s_c}\cdot\big|\hat
u(t,\xi)\big|^2d\xi\leq\eta.
\end{equation}
We refer to the function $N(t)$ as the frequency scale function for
the solution $u$, to $x(t)$ as the spatial center function, and to
$C(\eta)$ as the compactness modules function.
\end{definition}
\begin{remark}\label{rem1.1}
By Ascoli-Arzela Theorem, $u$ is almost periodic modulo symmetries
if and only if the set
$$\Big\{N(t)^{s_c-\frac{d}2}u\Big(t,x(t)+\frac{x}{N(t)}\Big),~t\in
I\Big\}$$ falls in a compact set in $\dot{H}^{s_c}_x(\R^d)$. The
following are  consequences of this statement. If $u$ is almost
periodic modulo symmetries, then
 there exists $c(\eta)>0$ such that
\begin{equation}\label{apss1}
\int_{|x-x(t)|\leq\frac{c(\eta)}{N(t)}}\big||\nabla|^{s_c}
u(t,x)\big|^2dx\leq\eta
\end{equation}
and
\begin{equation}\label{xiaoc}
\int_{|\xi|\leq c(\eta)N(t)}|\xi|^{2s_c}\cdot|\hat
u(t,\xi)|^2d\xi\leq\eta.
\end{equation}
\end{remark}

By the same argument as in \cite{MiaoXu,Pau}, we can show that if
Theorem \ref{theorem} fails, then we will inevitably encounter at
least one of the following three enemies.
\begin{theorem}[Three enemies, \cite{MiaoXu,Pau}]\label{three} Suppose $d\geq9$ is
such that Theorem \ref{theorem} fails, that is, $E_{c}<+\infty.$
Then there exists a maximal-lifespan solution $u:~I\times\R^d\to\C$,
which is almost periodic modulo symmetries, with $S_I(u)=+\infty.$
Furthermore, we can also ensure that the lifespan $I$ and the
frequency scale function $N(t):~I\to\R^+$ satisfy one of the
following three scenarios:
\begin{enumerate}
\item $($Finite time blowup$)$ Either $|\inf(I)|<+\infty$ or $\sup(I)<+\infty.$

\item $($Soliton-like solution$)$ $I=\R$ and $N(t)=1$ for all $t\in\R.$

\item $($Low-to-high frequency cascade$)$ $I=\R$,
$$\inf_{t\in\R}N(t)\geq1,~\text{and}~\varlimsup_{t\to\infty}N(t)=+\infty.$$
\end{enumerate}
\end{theorem}

 In view of this theorem, our
goal is to preclude  the possibilities of all the scenarios.

We also need the following Duhamel formula, which is important for
showing the additional decay and negative regularity in Section 4.
This is a robust consequence of almost periodicity modulo
symmetries; see, for example, \cite{CKSTT07,MiaoXuZhao,Pau}.

\begin{lemma} [No-waste Duhamel formula] Let $u: I\times\R^d\to\C$ be a maximal-lifespan solution which is
almost periodic modulo symmetries. Then, for all $t\in I,$ there
holds that
\begin{equation}\label{nwd}
\begin{split}
u(t)=&\lim_{T\nearrow\sup(I)}i\int_t^{T}e^{i(t-s)\Delta^2}f(u)(s)ds\\
=&-\lim_{T\searrow\inf(I)}i\int_T^te^{i(t-s)\Delta^2}f(u)(s)ds
\end{split}
\end{equation}
as weak limits in $\dot{H}_x^{s_c}(\R^d)$.
\end{lemma}

With this lemma in hand,  we can deduce that the energy or mass of
the finite time blow-up solution is zero and so get a contradiction.
We refer to Section 3 for more details.

In view of the no-waste Duhamel formula and noting that the minimal
$L_t^\infty\dot{H}^{s_c}_x$-norm blowup solution is localized in
both physical and frequency space, we will show that it admits lower
regularity.

\begin{theorem}[Negative regularity in the global case]\label{regular}  Let
$u$ be a global solution to \eqref{equ1.1} which is almost periodic
modulo symmetries in the sense of Theorem \ref{three}. And assume
that $\inf\limits_{t\in\R}N(t)\geq1$, then there exists a constant
$\al>0$ such that for any $0<\varepsilon<\al$
\begin{equation}
u\in L_t^\infty(\R;~\dot H^{-\varepsilon}_x(\R^d)).
\end{equation}
\end{theorem}

Combining this theorem with interaction Morawetz estimate and
interpolation, we will get a contradiction for the global almost
periodic solutions in the sense of Theorem \ref{three}. Thus, we
conclude Theorem \ref{theorem}. We refer to Section 3 for more
details.

\vskip 0.2in
 The paper is organized as follows. In Section $2$, we
deal with the local theory
 for the equation \eqref{equ1.1}.  In Section $3$, we exclude three scenarios in the sense of Theorem
\ref{three} under the assumption that Theorem \ref{regular} holds.
In Section $4$, we show the global solutions which are almost
periodic modulus symmetries   admit the negative regularity, that
is, Theorem \ref{regular}. Hence we conclude the proof of Theorem
\ref{theorem}. Finally, we show the stability in Appendix.

\subsection{Notations}
Finally, we conclude the introduction by giving some notations which
will be used throughout this paper. To simplify the expression of
our inequalities, we introduce some symbols $\lesssim, \thicksim,
\ll$. If $X, Y$ are nonnegative quantities, we use $X\lesssim Y $ or
$X=O(Y)$ to denote the estimate $X\leq CY$ for some $C$ which may
depend on the critical energy $E_{c}$  but not on any parameter such
as $\eta$  and $\rho$, and $X \thicksim Y$ to denote the estimate
$X\lesssim Y\lesssim X$. We use $X\ll Y$ to mean $X \leq c Y$ for
some small constant $c$ which is again allowed to depend on $E_{c}$.
We use $C\gg1$ to denote various large finite constants, and $0< c
\ll 1$ to denote various small constants. Any summations over
capitalized variables such as $M_j$ are presumed to be dyadic, i.e.,
these variables range over numbers of the form $2^k$ for $k\in
\mathbb{Z}$. For any $r, 1\leq r \leq \infty$, we denote by $\|\cdot
\|_{r}$ the norm in $L^{r}=L^{r}(\mathbb{R}^d)$ and by $r'$ the
conjugate exponent defined by $\frac{1}{r} + \frac{1}{r'}=1$. We
denote  $a\pm$ to be any quantity of the form $a\pm\epsilon$ for any
$\epsilon>0.$

The Fourier transform on $\mathbb{R}^d$ is defined by
\begin{equation*}
\aligned \widehat{f}(\xi):= \big( 2\pi
\big)^{-\frac{d}{2}}\int_{\mathbb{R}^d}e^{- ix\cdot \xi}f(x)dx ,
\endaligned
\end{equation*}
giving rise to the fractional differentiation operators
$|\nabla|^{s}$ and $\langle\nabla\rangle^s$,  defined by
\begin{equation*}
\aligned
\widehat{|\nabla|^sf}(\xi):=|\xi|^s\hat{f}(\xi),~~\widehat{\langle\nabla\rangle^sf}(\xi):=\langle\xi\rangle^s\hat{f}(\xi),
\endaligned
\end{equation*} where $\langle\xi\rangle:=1+|\xi|$.
This helps us to define the homogeneous and inhomogeneous Sobolev
norms
\begin{equation*}
\big\|f\big\|_{\dot{H}^s_x(\R^d)}:= \big\|
|\xi|^s\hat{f}\big\|_{L^2_x(\R^d)},~~\big\|f\big\|_{{H}^s_x(\R^d)}:=
\big\| \langle\xi\rangle^s\hat{f}\big\|_{L^2_x(\R^d)}.
\end{equation*}

We will also need the Littlewood-Paley projection operators.
Specifically, let $\varphi(\xi)$ be a smooth bump function adapted
to the ball $|\xi|\leq 2$ which equals 1 on the ball $|\xi|\leq 1$.
For each dyadic number $N\in 2^{\mathbb{Z}}$, we define the
Littlewood-Paley operators
\begin{equation*}
\aligned \widehat{P_{\leq N}f}(\xi)& :=
\varphi\Big(\frac{\xi}{N}\Big)\widehat{f}(\xi), \\
\widehat{P_{> N}f}(\xi)& :=
\Big(1-\varphi\Big(\frac{\xi}{N}\Big)\Big)\widehat{f}(\xi), \\
\widehat{P_{N}f}(\xi)& :=
\Big(\varphi\Big(\frac{\xi}{N}\Big)-\varphi\Big(\frac{2\xi}{N}\Big)\Big)\widehat{f}(\xi).
\endaligned
\end{equation*}
Similarly we can define $P_{<N}$, $P_{\geq N}$, and $P_{M<\cdot\leq
N}=P_{\leq N}-P_{\leq M}$, whenever $M$ and $N$ are dyadic numbers.
We will frequently write $f_{\leq N}$ for $P_{\leq N}f$ and
similarly for the other operators.

The Littlewood-Paley operators commute with derivative operators,
the free propagator, and the conjugation operation. They are
self-adjoint and bounded on every $L^p_x$ and $\dot{H}^s_x$ space
for $1\leq p\leq \infty$ and $s\geq 0$, moreover, they also obey the
following
 Bernstein estimates

\begin{lemma}[Bernstein estimates]
\begin{eqnarray*}\label{bernstein}
 \big\| P_{\geq N} f \big\|_{L^p} & \lesssim & N^{-s} \big\|
|\nabla|^{s}P_{\geq N} f \big\|_{L^p}, \\
\big\||\nabla|^s P_{\leq N} f \big\|_{L^p} & \lesssim  & N^{s}
\big\|
P_{\leq N} f \big\|_{L^p},  \\
\big\||\nabla|^{\pm s} P_{N} f \big\|_{L^p} & \thicksim & N^{\pm s}
\big\|
P_{N} f \big\|_{L^p},  \\
\big\| P_{\leq N} f \big\|_{L^q} & \lesssim &
N^{\frac{d}{p}-\frac{d}{q}} \big\|
P_{\leq N} f \big\|_{L^p},  \\
\big\| P_{ N} f \big\|_{L^q} & \lesssim &
N^{\frac{d}{p}-\frac{d}{q}} \big\|P_{ N} f \big\|_{L^p},
\end{eqnarray*}
where  $s\geq 0$ and $1\leq p\leq q \leq \infty$.

 \end{lemma}

\section{Preliminaries}
 \setcounter{section}{2}\setcounter{equation}{0}

 \subsection{Strichartz estimate and nonlinear estimates}
In this section, we consider the Cauchy problem  for fourth-order
Schr\"odinger equation
\begin{equation} \label{equ2}
    \left\{ \aligned &iu_t+\Delta^2u-f(u)=  0, \\
    &u(0)=u_0.
    \endaligned
    \right.
\end{equation}
The integral equation for the Cauchy problem $(\ref{equ2})$ can be
written as
\begin{equation}\label{inte1}
u(t,x)=e^{i(t-t_0)\Delta^2}u(t_0)-i\int_{t_0}^te^{i(t-s)\Delta^2}f(u(s))ds.
\end{equation}

The biharmonic Schr\"{o}dinger semigroup is defined for any tempered
distribution $g$ by
\begin{equation*}
e^{it\Delta^{2}}g=\mathcal{F}^{-1}e^{it|\xi|^{4}}\mathcal{F}g.
\end{equation*}

Now we recall the dispersive estimate for the biharmonic
Schr\"odinger operator.

\begin{lemma}[Dispersive estimate, \cite{Ben-Artzi}]\label{disp} Let $2\leq q\leq+\infty.$ Then, we have
the following dispersive estimate
\begin{equation}\label{dispers}
\big\|e^{it\Delta^2}f \big\|_{L^q_x(\R^d)} \leq
C|t|^{-\frac{d}4(1-\frac2{q})} \|f\|_{L^{q'}_x(\R^d)}
\end{equation}
for all $t\neq0$ and $2\leq q\leq+\infty,~\frac1q+\frac1{q'}=1$.
\end{lemma}

The Strichartz estimates involve the following definitions:
\begin{definition}\label{def1}
A pair of Lebesgue space exponents $(q,r)$ are called
Schr\"{o}dinger admissible for $\mathbb{R}^{1+d}$, or denote by
$(q,r)\in \Lambda_{0}$ when $q,r\geq 2,~(q,r,d)\neq (2,\infty,2)$,
and
\begin{equation}\label{equ21}
\frac{2}{q}=d\Big(\frac{1}{2}-\frac{1}{r}\Big).
\end{equation}

\end{definition}

\begin{definition}\label{def2}
In addition, a pair of Lebesgue space exponents $(\gamma,\rho)$ are called
biharmonic admissible for $\mathbb{R}^{1+d}$ or denote by
$(\gamma,\rho)\in\Lambda_{1}$when $\gamma,\rho\geq
2,~(\gamma,\rho,d)\neq(2,\infty,4)$, and
\begin{equation}\label{equ22}
\frac{4}{\gamma}=d\Big(\frac{1}{2}-\frac{1}{\rho}\Big).
\end{equation}
For a fixed spacetime slab $I\times\R^d$, we define the Strichartz
norm
$$\|u\|_{S^0(I)}:=\sup_{(q,r)\in\Lambda_1}\|u\|_{L_t^q
L_x^r(I\times\R^d)}.$$ We denote $S^0(I)$ to be the closure of all
test functions under this norm and write $N^0(I)$ for the dual of
$S^0(I)$.
\end{definition}

According to the above dispersive estimate, the abstract duality and
interpolation argument(see \cite{KeT98}), we have the following
Strichartz estimates.

\begin{proposition}[Strichartz estimates for Fourth-order Schr\"{o}dinger\cite{MiaoZhang, Pausader}]\label{prop1}
Let $s\geq 0$, suppose that $u(t, x)$ is a solution on $[0,T]$ to
the initial value problem
\begin{align}\label{equ23}
\begin{cases}
(i\partial_{t}+\Delta^{2})u(t,x)=h, \quad (t,x)\in [0,T]\times \mathbb{R}^{d}\\
u(0)=u_{0}(x),
\end{cases}
\end{align}
for some data $u_{0}$ and $T > 0$. Then we have the Strichartz
estimate, for $(q, r), (a, b)\in \Lambda_{0}$
\begin{equation}\label{equ24}
\big\||\nabla|^{s}u\big\|_{L^{q}([0,T];L^{r}(\mathbb{R}^d))}\lesssim\big\||\nabla
|^{s-\frac{2}{q}}u_{0}\big\|_{L^{2}(\mathbb{R}^d)}+
\big\||\nabla|^{s-\frac{2}{q}-\frac{2}{a}}h\big\|_{L^{a^{\prime}}([0,T];L^{b^{\prime}}(\mathbb{R}^d))},
\end{equation}
and for $(\gamma,\rho),(c,d)\in \Lambda_{1}$
\begin{equation}{\label{equ25}}
\|u\|_{L^{\gamma}([0,T];L^{\rho}(\mathbb{R}^d))}\lesssim\|u_{0}\|_{L^{2}(\mathbb{R}^d)}+
\|h\|_{L^{c^{\prime}}([0,T];L^{d^{\prime}}(\mathbb{R}^d))}.
\end{equation}
In particular, we have
\begin{equation}\label{equ26}
\big\||\nabla|^su\big\|_{S^0(I)}\lesssim\big\||\nabla|^su_{0}\big\|_{L^{2}(\mathbb{R}^d)}+
\big\||\nabla|^sh_1\big\|_{N^0(I)}+\big\||\nabla|^{s-1}h_2\big\|_{L_t^2L_x^\frac{2d}{d+2}(I\times\R^d)},
\end{equation}
where $h=h_1+h_2$ and $I=[0,T].$
\end{proposition}

The key feature of such lemma is that the spacetime-norm of the
$s$-derivative of $u$ can be estimated by  $(s-1)$-derivative of the
forcing term, which is the consequence of smoothing effect for all
higher-order nonlinear Schr\"odinger equations, see Proposition 2 in
\cite{MZ}. This enables us to consider $s<2+p$ for $p$ being not
even integer. This is in sharp contrast with NLS, where the
restriction for the regularity $s<1+p$ when $p$ is not an even
integer.

Now we give a few nonlinear estimates which will be applied  to show
the local well-posedness which is the first step to obtain the
global time-space estimate that leads to the scattering.

\begin{lemma}
$(i)$ $($Product rule$)$ Let $s\geq0$, and $1<r,p_j,q_j<\infty$ such
that $\frac1r=\frac1{p_i}+\frac1{q_i}~(i=1,2).$ Then, we have
\begin{equation}\label{moser}
\big\||\nabla|^s(fg)\big\|_{L_x^r(\R^d)}\lesssim\|f\|_{{L_x^{p_1}(\R^d)}}\big\||\nabla|^sg
\big\|_{{L_x^{q_1}(\R^d)}}+\big\||\nabla|^sf\big\|_{{L_x^{p_2}(\R^d)}}\|g\|_{{L_x^{q_2}(\R^d)}}.
\end{equation}

$(ii)$ $(C^1$ continuous$)$ Assume that $G\in
C^1(\mathbb{C}),~s\in(0,1],~1<p,p_1,p_2<+\infty,~\frac1p=\frac1{p_1}+\frac1{p_2}.$
Then, we have
\begin{equation}\label{fraclsfz}
\big\||\nabla|^s
G(u)\big\|_p\lesssim\|G'(u)\|_{p_1}\big\||\nabla|^su\big\|_{p_2}.
\end{equation}

$(iii)$ $($H\"older continuous$)$ Let $G\in C^\al(\mathbb{C})$ with
$0<\al<1$. Then, for every
$0<s<\al,~1<p<+\infty,~\frac{s}{\al}<\sigma<1,$ we have
\begin{equation}\label{hfscls}
\big\||\nabla|^s
G(u)\big\|_p\lesssim\big\||u|^{\al-\frac{s}{\sigma}}\big\|_{p_1}\big\||\nabla|^\sigma
u\big\|_{\frac{s}{\sigma}p_2}^{\frac{s}{\sigma}},
\end{equation} provided
$\frac1p=\frac{1}{p_1}+\frac1{p_2},$ and
$\big(1-\frac{s}{\al\sigma}\big)p_1>1.$
\end{lemma}
\begin{proof}
We refer to \cite{CW,Visanphd} for the proof.
\end{proof}

As a direct consequence, we obtain the following nonlinear estimate.

\begin{corollary}\label{fscqdfz}
Let $f(u)=|u|^pu,$ and let $s\geq0$ if $p$ is an even integer or
$0\leq s<1+p$ otherwise. Then, we have
\begin{equation}\label{fxxqd}
\big\||\nabla|^sf(u)\big\|_{L_x^q}\lesssim\big\||\nabla|^su\big\|_{L_x^{q_1}}\|u\|_{L_x^{q_2}}^p,
\end{equation}
where $\frac1q=\frac1{q_1}+\frac{p}{q_2}.$
\end{corollary}

We will also make use of the following refinement of the fractional
chain rule, which appears in \cite{KV2011}. This will be used in the
proof of the perturbation for $s_c\in[1,2)$.
\begin{lemma}[Derivatives of differences,
\cite{KV2011}]\label{deridiff}  For $0<s<1$ and $f(u)=|u|^pu$. Then
for $1<r,r_1,r_2<+\infty$ such that
$\frac1r=\frac1{r_1}+\frac{p}{r_2}$, we have
\begin{equation}
\big\||\nabla|^s\big[f(u+v)-f(u)\big]\big\|_r\lesssim\big\||\nabla|^su\big\|_{r_1}\|v\|_{r_2}^p+\big\||\nabla|^sv\big\|_{r_1}\|u+v\|_{r_2}^p.
\end{equation}
\end{lemma}

Next, we give a nonlinear estimate in \cite{KV2010}. It is used in
the proof of Lemma \ref{nlefs}.
\begin{lemma}[\cite{KV2010}]\label{yashuo}
Let $G\in C^\al(\mathbb{C})$ with $0<\al\leq1$., and
$0<s<\sigma\al<\al$. For $1<q,q_1,q_2,r_1,r_2,r_3<+\infty,$ such
that
$\frac1q=\frac1{q_1}+\frac1{q_2}=\frac1{r_1}+\frac1{r_2}+\frac1{r_3},$
we have
\begin{equation}\label{yashuo1}
\begin{split}
&\bigg\||\nabla|^s\Big[\omega\cdot\big(G(u+v)-G(u)\big)\Big]\bigg\|_q\\
\lesssim&\big\||\nabla|^s\omega\big\|_{q_1}\|v\|_{\al
q_2}^p+\|\omega\|_{r_1}\|v\|_{(\al-\frac{s}{\sigma})r_2}^{\al-\frac{s}{\sigma}}\Big(\big\||\nabla|^{\sigma}v\big\|_{\frac{s}{\sigma}r_3}+\big\||\nabla|
^\sigma u\big\|_{\frac{s}{\sigma}r_3}\Big)^\frac{s}{\sigma},
\end{split}
\end{equation}
where $(1-\al)r_1,~\big(\al-\frac{s}{\sigma}\big)r_2>1$.
\end{lemma}
We remark that one can extend Lemma \ref{yashuo} to $G(u)\simeq
O\big(|u|^\al\big)$ with $\al>1$, which will be used in the proof of
\eqref{nlesf11} for $p>1$.

We will also need the following lemma which is similar to Lemma
\ref{fraclsfz}. It is useful to the proof of Proposition
\ref{adecay1} for $p<1$.
\begin{lemma}[Nonlinear Bernstein inequality \cite{KVnote}]\label{nlbern}
Assume that $G\in C^\al(\mathbb{C})$ with $0<\al\leq1$.  Then,  we
have
\begin{equation}\label{nlbernstein}
\big\|P_NG(u)\big\|_{L_x^\frac{q}{\al}(\R^d)}\lesssim
N^{-\al}\big\|\nabla u\big\|_{L_x^q(\R^d)}^\al
\end{equation}
for all $1\leq q<+\infty.$
\end{lemma}

\subsection{Local well-posedness in inhomogeneous space}

Now we can state the following standard local well-posedness result,
where we assume that the initial data in the inhomogeneous critical
Sobolev space. This assumption simplifies the proof since one can
use the $L_t^q L_x^r$-norm with $(q,r)\in\Lambda_1$ as the metric
(that is in mass-critical spaces) when we prove the map is a
contraction. And this assumption can be removed by using the
perturbation results proved in Corollary \ref{long1} below, see
Proposition \ref{lwph}.

\begin{theorem} [Local well-posedness]\label{feiqi}
Assume $u_0\in
 H_x^{s_c}(\R^d)$, and let $s_c\geq1$ if $p$ is an even integer or
$s_c\in[1,2+p)$ otherwise. Then there exists $\eta_0=\eta_0(d)>0$
such that if $I$ is a compact interval containing zero such that
\begin{equation}\label{small}
\|e^{it\Delta^2}u_0\|_{Z(I)}:=\big\||\nabla|^{s_c-1}e^{it\Delta^2}u_0\big\|_{L_t^{2(p+1)}L_x^\frac{2d(p+1)}{(d-2)(p+1)-4}(I\times\R^d)}\leq\eta,
\end{equation}
where $0<\eta\leq\eta_0$, then there exists a unique solution $u$ to
\eqref{equ1.1} on $I\times\R^d$. Furthermore, the solution $u$ obeys
\begin{align}\label{smalll}
\|u\|_{Z(I)}\leq&2\eta\\\label{small1}
\big\||\nabla|^{s_c}u\big\|_{S^0(I)}\leq&2C\big\||\nabla|^{s_c}u_0\big\|_{L_x^2}+C\eta^{1+p}\\\label{small2}
\|u\|_{S^0(I)}\leq&2C\|u_0\|_{L_x^2},
\end{align}
where $C$ is the Strichartz constant as in  Proposition \ref{prop1}.
\end{theorem}

\begin{proof}
We apply the Banach fixed point argument to prove this lemma. First
we define the map
\begin{equation}\label{inte3}
\Phi(u(t))=e^{it\Delta^2}u_0-i\int_{0}^te^{i(t-s)\Delta^2}f(u(s))ds
\end{equation}
on the complete metric space $B$
\begin{align*}
B:=\big\{u\in C_t(I; H_x^{s_c}):&\
\|u\|_{L_t^\infty H_x^{s_c}(I\times\R^d)}\leq2C\|u_0\|_{H_x^{s_c}}+C\eta^{1+p};\\
&\|u\|_{L_t^{2(p+1)}L_x^\frac{2d(p+1)}{d(p+1)-4}(I\times\R^d)}\leq2C\|u_0\|_{L_x^2};~\|u\|_{Z(I)}\leq2\eta\big\}
\end{align*}
with the metric
$d(u,v)=\big\|u-v\big\|_{L_t^{2(p+1)}L_x^\frac{2d(p+1)}{d(p+1)-4}(I\times\R^d)}$.

It suffices to prove that the operator defined by the RHS of
$(\ref{inte3})$ is a contraction map on $B$ for $I$. If $u\in B,$
then by Strichartz estimate,  Corollary \ref{fscqdfz} and
\eqref{small}, we have
\begin{align*}\|\Phi(u)\|_{Z(I)}
=&\big\||\nabla|^{s_c-1}\Phi(u)\big\|_{L_t^{2(p+1)}L_x^\frac{2d(p+1)}{(d-2)(p+1)-4}(I\times\R^d)}\\
\leq& \|e^{it\Delta^2}u_0\|_{Z(I)}+C\big\||\nabla|^{s_c-1}f(u)
\big\|_{L_t^2L_x^\frac{2d}{d+2}(I\times\R^d)
}\\
\leq&\eta+C\|u\|_{Z(I)}^{p+1}.
\end{align*}
Plugging the assumption $\|u\|_{Z(I)}\leq2\eta$, we see that for
$u\in B$,
\begin{align*}
\|\Phi(u)\|_{Z(I)}\leq\eta+8C\eta^{p+1}\leq2\eta
\end{align*}
provided we take $\eta$ sufficiently small such that
$8C\eta^p\leq1.$ Similarly, if $u\in B,$ then
\begin{align*}
&\|\Phi(u)\|_{L_t^\infty H_x^{s_c}(I\times\R^d)}\\
 \leq&C\|u_0\|_{
H^{s_c}_x(\R^d)}+C\big\||\nabla|^{s_c-1}f(u)\big\|_{L_t^2L_x^\frac{2d}{d+2}(I\times\R^d)
}+C\|f(u)\|_{L_t^2L_x^\frac{2d}{d+4}(I\times\R^d)}\\
\leq&C\|u_0\|_{
H^{s_c}_x(\R^d)}+C\|u\|_{Z(I)}\|u\|_{L_t^{2(p+1)}L_x^\frac{dp(p+1)}{2(p+2)}(I\times\R^d)}^p+C
\|u\|_{L_t^{2(p+1)}L_x^\frac{2d(p+1)}{d(p+1)-4}(I\times\R^d)}\|u\|_{Z(I)}^p\\
\leq&C\|u_0\|_{
H^{s_c}_x(\R^d)}+C(2\eta)^{p+1}+2C\|u_0\|_{L_x^2}(2\eta)^p\\
\leq&2C\|u_0\|_{ H^{s_c}_x(\R^d)}+C\eta^{1+p},
\end{align*}
and
\begin{align*}
\big\|\Phi(u)\big\|_{L_t^{2(p+1)}L_x^\frac{2d(p+1)}{d(p+1)-4}(I\times\R^d)}
 \leq&C\|u_0\|_{L^2_x(\R^d)}+C\big\|f(u)\big\|_{L_t^2L_x^\frac{2d}{d+4}(I\times\R^d)
}\\
\leq&C\|u_0\|_{
L^2_x(\R^d)}+C\|u\|_{L_t^{2(p+1)}L_x^\frac{2d(p+1)}{d(p+1)-4}(I\times\R^d)}\|u\|_{Z(I)}^p\\
\leq&C\|u_0\|_{
L^2_x(\R^d)}+C^2\|u_0\|_{L^2_x(\R^d)}(2\eta)^p\\
\leq&2C\|u_0\|_{L^2_x(\R^d)}.
\end{align*}
Hence $\Phi(u)\in B$ for $u\in B.$ That is, the functional $\Phi$
 maps the set $B$ back to itself.

On the other hand, by the same argument as before, we have for $u,
v\in B$,
\begin{align*}
d\big(\Phi(u),\Phi(v)\big)
\leq&C\big\|f(u)-f(v)\big\|_{L_t^2L_x^\frac{2d}{d+4}(I\times\R^d)}\\
\leq&C\|u-v\|_{L_t^{2(p+1)}L_x^\frac{2d(p+1)}{d(p+1)-4}(I\times\R^d)}\|(u,v)\|_{Z(I)}^p\\
\leq&C(4\eta)^pd(u,v)
\end{align*}
which allows us to derive
\begin{equation*}
d\big(\Phi(u),\Phi(v)\big)\leq\frac{1}{2}d(u,v),
\end{equation*}
by taking $\eta$ small such that
$$C(4\eta)^p\leq\frac12.$$

A standard fixed point argument gives a unique solution $u$ of
\eqref{equ1.1} on $I\times\R^d$ which satisfies the bound
\eqref{smalll}. The bounds \eqref{small1} and \eqref{small2} follow
from another application of the Strichartz estimate.
\end{proof}

\subsection{Perturbation}
Closely related to the continuous dependence on the data, an
essential tool for concentration compactness
 arguments is the perturbation theory. And we will show this perturbation theory in Appendix.

\begin{lemma}[Perturbation Lemma]\label{pertu123}  Let $s_c\geq1.$  Assume in addition that $s_c<2+p$ if $p$ is not an even integer. Let $I$ be
a compact time interval and $u,~\tilde{u}$ satisfy
\begin{align*}
(i\partial_t+\Delta^2)u=&-f(u)+eq(u)\\
(i\partial_t+\Delta^2)\tilde{u}=&-f(\tilde{u})+eq(\tilde{u})
\end{align*}
for some function $eq(u),eq(\tilde{u})$, and $f(u)=|u|^pu$. Assume
that for some constants $M,E>0$, we have
\begin{align}\label{eq2.2012}
\|u\|_{L_t^\infty(I;
\dot{H}^{s_c}_x(\R^d))}+\|\tilde{u}\|_{L_t^\infty(I;
\dot{H}^{s_c}_x(\R^d))}\leq E,\\\label{equ2.20122}
S_I(\tilde{u})\leq M,
\end{align}
 Let $t_0\in I$, and let $u(t_0)$ be
close to $\tilde{u}(t_0)$ in the sense that
\begin{equation} \label{eq2.2221}
\|u_0-\tilde{u}_0\|_{\dot{H}^{s_c}_x(\R^d)}\leq\varepsilon,
\end{equation}
where $0<\varepsilon<\varepsilon_1( M, E)$ is a small constant.
Assume also that we have smallness conditions
\begin{equation}\label{equ2.2121}
\big\||\nabla|^{s_c-1}\big(eq(u),eq(\tilde{u})\big)\big\|_{L_t^2L_x^\frac{2d}{d+2}(I\times\R^d)}\leq\varepsilon,
\end{equation}
where $\varepsilon$ is as above.

 Then
we conclude that
\begin{equation}\label{eq2.2321}
\begin{split}
S_I(u-\tilde{u})\leq & C(M,E)\varepsilon^{c_1}\\
\big\||\nabla|^{s_c}(u-\tilde{u})\big\|_{S^0(I)}\leq &
C(M,E)\varepsilon^{c_2}\\
\big\||\nabla|^{s_c}u\big\|_{S^0(I)}\leq & C(M,E),
\end{split}
\end{equation}
where $c_1,~c_2$ are positive constants that depend on $d,~p,~E$ and
$M$.
\end{lemma}

\subsection{Local well-posedness in homogenous space and stability}
As stated in the subsection 2.2, the assumption that the initial
data in the inhomogeneous critical Sobolev space can be removed by
the perturbation results. Now we give a detail proof.
\begin{proposition} [Local well-posedness in homogenous space]\label{lwph} Assume that $s_c\geq1$ if $p$ is an even
integer or $1\leq s_c<2+p$ otherwise. Let $u_0\in \dot
H_x^{s_c}(\R^d)$. Then, if  $I$ is a compact interval containing
zero such that
\begin{equation}\label{small1231}
\|e^{it\Delta^2}u_0\|_{Z(I)}:=\big\||\nabla|^{s_c-1}e^{it\Delta^2}u_0\big\|_{L_t^{2(p+1)}L_x^\frac{2d(p+1)}{(d-2)(p+1)-4}(I\times\R^d)}\leq\frac\eta2,
\end{equation}
where $\eta$ is as in Theorem \ref{feiqi}, then there exists a
unique solution $u$ to \eqref{equ1.1} on $I\times\R^d$. Furthermore,
the solution $u$ satisfies the bounds
\begin{align}\label{smalll11}
\|u\|_{Z(I)}\leq&2\eta\\\label{small123}
\big\||\nabla|^{s_c}u\big\|_{S^0(I)}\leq&2C\big\||\nabla|^{s_c}u_0\big\|_{L_x^2}+C\eta^{1+p},
\end{align}
where $C$ is the Strichartz constant as in  Proposition \ref{prop1}.

In particular, if $\|u_0\|_{\dot{H}^{s_c}_x(\R^d)}\leq\frac\eta2,$
then the solution $u$ is global and scatters.
\end{proposition}

\begin{proof}
Since $H^{s_c}(\R^d)$ is dense in $\dot{H}^{s_c}(\R^d),$
we know that for any  $u_0\in\dot{H}^{s_c}(\R^d),$ there exists a
sequence $\{u_n(0)\}\subset H^{s_c}(\R^d)$ such that
$$\|u_n(0)-u_0\|_{\dot{H}^{s_c}(\R^d)}\to 0,~\text{as}\quad
n\to+\infty.$$ Hence,
$\forall~\varepsilon>0,~\exists~N>0,~s.t.~\forall~n>N,$
$$\|u_n(0)-u_0\|_{\dot{H}^{s_c}(\R^d)}<\varepsilon.$$
By Strichartz estimate and \eqref{small1231}, we get for
$2C\varepsilon<\eta$ and $n>N,$
\begin{align*}
\big\|e^{it\Delta^2}u_n(0)\big\|_{Z(I)} \leq
C\|u_n(0)-u_0\|_{\dot{H}^{s_c}(\R^d)}+
\big\|e^{it\Delta^2}u_0\big\|_{Z(I)} \leq
C\varepsilon+\frac\eta2\leq\eta.
\end{align*}
This together with $u_n(0)\in H^{s_c}(\R^d),$ and Theorem
\ref{feiqi} yield that there exists a unique solution
$u_n(t,x):~I\times\R^d\to \mathbb{C}$ to \eqref{equ1.1} with initial
data $u_n(0)$ obeying \eqref{smalll}-\eqref{small2}. In particular,
it satisfies
\begin{equation}\label{xxundz}
\big\||\nabla|^{s_c}u_n\big\|_{S^0(I\times\R^d)}\lesssim\big\||\nabla|^{s_c}u_n(0)\big\|_{L_x^2}+\eta^{1+p}
\lesssim\big\||\nabla|^{s_c}u_0\big\|_{L_x^2}+\eta^{1+p}+\varepsilon.
\end{equation}

Next we use Lemma \ref{pertu123} to show the solution sequence
 $\{u_n(t,x)\}$ is Cauchy in
 $S^{s_c}(I)$, where $\|u\|_{S^{s_c}(I)}:=\big\||\nabla|^{s_c}u\big\|_{S^0(I)}.$
  In fact, it follows from  Lemma \ref{pertu123} if we set $\tilde{u}=u_m,~u=u_n,$ and $eq(u)=eq(\tilde{u})=0.$
Thus, by \eqref{eq2.2321}, we get
$$\big\||\nabla|^{s_c}(u_n-u_m)\big\|_{S^0(I)}\leq C(E,M)\varepsilon,$$ which means $\{u_n(t,x)\}$ is Cauchy in $S^{s_c}(I)$.
And so it convergent to a solution   $u(t,x)$ with initial data
$u(0,x)=u_0$ obeying $|\nabla|^{s_c}u\in S^0(I).$
\end{proof}

\vskip 0.2in Using the Theorem \ref{feiqi} and Lemma \ref{pertu123}
as well as their proof, one easily derives the following local
theory for \eqref{equ1.1}. We refer the author to Pausader
\cite{Pausader} for the special energy-critical case ($s_c=2$).

\begin{theorem}\label{lwp}
Let $s_c\geq1.$  Assume in addition that $s_c<2+p$ if $p$ is not an
even integer. Then, given $u_0\in \dot H_x^{s_c}(\R^d)$ and
$t_0\in\R$, there exists a unique maximal-lifespan solution $u:
I\times\R^d\to\C$ to \eqref{equ1.1} with initial data $u(t_0)=u_0$.
This solution also has the following properties:
\begin{enumerate}
\item $($Local existence$)$ $I$ is an open neighborhood of $t_0$.
\item $($Blowup criterion$)$ If\ $\sup(I)$ is finite, then $u$ blows up
forward in time in the sense of Definition \ref{def1.2.1}. If $\inf
(I)$ is finite, then $u$ blows up backward in time.
\item $($Scattering$)$ If $\sup(I)=+\infty$ and $u$ does not blow up
forward in time, then $u$ scatters forward in time in the sense
\eqref{1.2.1}. Conversely, given $v_+\in \dot H^{s_c}(\R^d)$, there
is a unique solution to \eqref{equ1.1} in a neighborhood of infinity
so that \eqref{1.2.1} holds.
\end{enumerate}
\end{theorem}

It is easy to show the following stability result by Proposition
\ref{lwph} and Lemma \ref{pertu123} as well as their proof.
\begin{corollary}[stability]\label{long1}  Assume that $s_c\geq1$ if $p$ is an even integer or $1\leq s_c<2+p$ otherwise. Let $I$ be
a compact time interval containing zero and $\tilde{u}$ be an near
solution to \eqref{equ1.1} on $I\times\R^d$ in the sense that
\begin{align*}
i\tilde{u}_t+\Delta\tilde{u}-f(\tilde{u})+e=0
\end{align*}
for some function $e$. Assume that for some constants $M,E>0$, we
have
\begin{align}\label{eq2.201}
\|\tilde{u}\|_{L_t^\infty(I; \dot{H}^{s_c}_x(\R^d))}\leq
E,\\\label{equ2.201} S_I(\tilde{u})\leq M.
\end{align}
 Let  $u_0\in\dot{H}_x^{s_c}(\R^d)$ and assume the smallness
 conditions
\begin{align}\label{eq2.213}
\|u_0-\tilde{u}_0\|_{\dot{H}^{s_c}_x(\R^d)}+\big\||\nabla|^{s_c-1}e
\big\|_{L_t^2L_x^\frac{ 2d}{d+2}(I\times\R^d)}\leq \varepsilon
\end{align}
where $0<\varepsilon<\epsilon_1=\epsilon_1( M, E)$ is a small
constant. Then there exists a unique solution $u: I\times\R^d\to \C$
 to \eqref{equ1.1} with initial data $u_0$ at time $t=0$ obeying
\begin{align}\label{sta11}
S_I(u-\tilde{u})\leq & C(M,E)\varepsilon^{c_1}\\\label{sta21}
\big\||\nabla|^{s_c}(u-\tilde{u})\big\|_{S^0(I)}\leq &
C(M,E)\varepsilon^{c_2},\\\label{sta31}
\big\||\nabla|^{s_c}u\big\|_{S^0(I)}\leq & C(M,E).
\end{align}
where $c_1,~c_2$ are positive constants that depend on $d,p,E$ and
$M$.
\end{corollary}

\section{Extinction of three scenarios}
 \setcounter{section}{3}\setcounter{equation}{0}

In this section, we preclude three scenarios in the sense of Theorem
\ref{three} under the assumption that Theorem \ref{regular} holds.
We will prove Theorem \ref{regular} in the next section. First, we
preclude the finite time blowup solution by making use of No-waste
Duhamel formula.

\subsection{The finite blowup solution}
We argue by contradiction. Assume that there exists a solution $u:
I\times\R^d\to\C$ which is a finite time blowup in the sense of
Theorem \ref{three}. Assume also $T:=\sup(I)<+\infty,$ then, we have
by \eqref{assume1.1} and Sobolev embedding
\begin{equation}\label{sobolev}
\|u\|_{L_t^\infty
L_x^\frac{pd}4(I\times\R^d)}\lesssim\|u\|_{L_t^\infty\dot{H}^{s_c}(I\times\R^d)}\lesssim1.
\end{equation}

 First, we consider the
energy-subcritical and energy-critical case.

{\bf Case 1: $1\leq s_c\leq2$.} Using Strichartz estimate, Sobolev
embedding, \eqref{sobolev}, \eqref{nwd}
 and H\"older's inequality, we have
\begin{align*}\label{ngtive1}
\big\||\nabla|^{s_c-2}u(t)\big\|_{L_x^2}\leq&\Big\|\int_t^Te^{i(t-s)\Delta^2}|\nabla|^{s_c-2}f(u(s))ds\Big\|_{L_x^2}\\\nonumber
\lesssim&\big\||\nabla|^{s_c-2}f(u(s))\big\|_{L_t^2L_x^\frac{2d}{d+4}([t,T)\times\R^d)}\\\nonumber
\lesssim&(T-t)^\frac12\|f(u)\|_{L_t^\infty L_x^\frac{pd}{4(p+1)}([t,T)\times\R^d)}\\\nonumber
\lesssim&(T-t)^\frac12\|u\|_{L_t^\infty
L_x^\frac{pd}4([t,T)\times\R^d)}^{p+1}\\\nonumber
\lesssim&(T-t)^\frac12.
\end{align*}
Interpolating this with $u\in L_t^\infty \dot{H}^{s_c}([0,T)\times\R^d)$, we deduce that
$$\|u(t)\|_{L_x^2}\lesssim (T-t)^{\frac{s_c}4}\to0,~\text{as}~t\to T$$
which shows that $u\in L_t^\infty L_x^2([0,T)\times\R^d)$ and also
$u\equiv0$ by the mass conservation. This contradicts with the fact
that $u$ is a blowup solution.

Next, we consider the energy-supercritical case. Using the
assumption \eqref{assume1.1} and Sobolev embedding, we have
\begin{equation}\label{sobolev1}
\|u\|_{L_t^\infty
L_x^\frac{pd}4(I\times\R^d)}\lesssim\big\||\nabla|^{s_c-2}u\big\|_{L_t^\infty
L_x^\frac{2d}{d-4}(I\times\R^d)}\lesssim\|u\|_{L_t^\infty\dot{H}^{s_c}(I\times\R^d)}\lesssim1.
\end{equation}

 {\bf Case 2:
$s_c\in(2,4].$} Combining \eqref{sobolev1} with No waste Duhamel
formula \eqref{nwd}, Strichartz estimate, H\"older's inequality and
Corollary \ref{fscqdfz}, one has
\begin{align}\label{ngtive}
\big\||\nabla|^{s_c-2}u(t)\big\|_{L_x^2}\leq&\Big\|\int_t^Te^{i(t-s)\Delta^2}|\nabla|^{s_c-2}f(u(s))ds\Big\|_{L_x^2}\\\nonumber
\lesssim&\big\||\nabla|^{s_c-2}f(u(s))\big\|_{L_t^2L_x^\frac{2d}{d+4}([t,T)\times\R^d)}\\\nonumber
\lesssim&(T-t)^\frac12\big\||\nabla|^{s_c-2}u\big\|_{L_t^\infty
L_x^\frac{2d}{d-4}([t,T)\times\R^d)}\|u\|_{L_t^\infty
L_x^\frac{pd}4([t,T)\times\R^d)}^p\\\nonumber
\lesssim&(T-t)^\frac12.
\end{align}
Interpolating this with \eqref{assume1.1}, we derive that
$$E(u_0)=E(u(t))\to 0,\quad \text{as}\quad t\to T,$$ which implies
that $u\equiv0.$ This contradicts with the fact that $u$ is a blowup
solution.

{\bf Case 3: $s_c\in(4,6]$.} It follows from  \eqref{ngtive} that
$u\in L_t^\infty([0,T); \dot{H}^{s_c-2}_x(\R^d)).$ Using  No waste
Duhamel formula \eqref{nwd}, Strichartz estimate, H\"older's
inequality and fractional chain rule, we obtain
\begin{align}\label{ngtive1}
\big\||\nabla|^{s_c-4}u(t)\big\|_{L_x^2}\leq&\Big\|\int_t^Te^{i(t-s)\Delta^2}|\nabla|^{s_c-4}f(u(s))ds\Big\|_{L_x^2}\\\nonumber
\lesssim&\big\||\nabla|^{s_c-4}f(u(s))\big\|_{L_t^2L_x^\frac{2d}{d+4}([t,T)\times\R^d)}\\\nonumber
\lesssim&(T-t)^\frac12\big\||\nabla|^{s_c-4}u\big\|_{L_t^\infty
L_x^\frac{2d}{d-4}([t,T)\times\R^d)}\|u\|_{L_t^\infty
L_x^\frac{pd}{4}([t,T)\times\R^d)}^p\\\nonumber
\lesssim&(T-t)^\frac12,
\end{align}
Interpolating this with \eqref{assume1.1} again, we also deduce that
$$E(u_0)=E(u(t))\to 0,\quad \text{as}\quad t\to T.$$
This contradicts with the fact that $u$ is a blowup
solution.

{\bf Case 4: $s_c\in(6,+\infty)$.} We can iterate the argument
presented above to obtain the contradiction.

Hence, we exclude the finite time blowup solution in the sense of
Theorem \ref{three}.

\subsection{The soliton-like solution}Next, we adopt the interaction Morawetz estimate to kill the
soliton-like solution.

We argue by contradiction. Assume that there exists a solution $u:
\R\times\R^d\to\C$ which is a soliton-like solution in the sense of
Theorem \ref{three}. Assume also Theorem \ref{regular} holds. In
particular, we have
\begin{equation}
u(t,x)\in L_t^\infty(\R; L_x^2(\R^d)).
\end{equation}

Therefore, the solution $u$ satisfies the following interaction
Morawetz estimate.
\begin{lemma}[Interaction Morawetz estimate, \cite{MiaoWuZhang, Pausader1}]\label{morawetz1}
Assume that $d\geq7.$ Let $u: \R\times\R^d\to\C$ be the solution to
\eqref{equ1.1}, and $u\in L_t^\infty(\R; {H}^\frac12_x(\R^d)).$
Then, for any compact interval $I\subset\R$, we have
\begin{equation}\label{mest1}
\int_I\iint_{\R^d\times\R^d}\frac{|u(t,x)|^{2}|u(t,y)|^2}{|x-y|^5}dxdydt\lesssim\|u\|_{L_t^\infty
L_x^2(I\times\R^d)}^2 \big\||\nabla_{x}|^\frac12u\|_{L_t^\infty
L^2_x(I\times\R^d)}^2\lesssim1.
\end{equation}
\end{lemma}
From \eqref{mest1}, we know that
$$\big\||\nabla|^{-\frac{d-5}2}(|u|^2)\big\|_{L_{t,x}^2(I\times\R^d)}\lesssim1.$$ And so, it follows from
\cite{Pausader1} that
\begin{equation}\label{djmr}
\begin{split}
 \big\||\nabla|^{-\frac{d-5}4}u\big\|_{
L^4_{t,x}(I\times\R^d)}\simeq&\Big\|\Big(\sum_{N\in2^{\Z}}N^{-\frac{d-5}2}|P_Nu|^2\Big)^\frac12\Big\|_{L_{t,x}^4(I\times\R^d)}\\
\lesssim&\big\||\nabla|^{-\frac{d-5}2}(|u|^2)\big\|_{
L^2_{t,x}(I\times\R^d)}^\frac12\lesssim1.
\end{split}
\end{equation}
Interpolating this with $u\in L_t^\infty(\R; \dot{H}_x^1(\R^d))$, we
obtain for all compact time interval $I\subset\R$
\begin{equation}\label{glmd}
\|u\|_{L_t^{d-1}
L_x^\frac{2(d-1)}{d-3}(I\times\R^d)}\lesssim1.
\end{equation}

Now we claim that
\begin{equation}\label{glmd1}
\|u\|_{L_x^\frac{2(d-1)}{d-3}(\R^d)}\gtrsim1,\quad \text{uniformly
for}\quad t\in\R.
\end{equation}
If this claim holds, then we derive a contradiction by taking the
length of the interval $I$ to be sufficiently large.

Hence it suffices to prove the claim \eqref{glmd1}. We argue by
contradiction. Suppose that the claim fails, then there exists a
time sequence $\{t_n\}$ such that $u(t_n)$ converges to zero in
$L_x^\frac{2(d-1)}{d-3}$. On the other hand, $u(t_n)$ converges
weakly to zero in $\dot{H}^{s_c}(\R^d)$ since $u(t)$ is uniformly
bounded in  $\dot{H}^{s_c}(\R^d)$. This contradicts with the fact
that the orbit of $u$ is precompact in $\dot{H}^{s_c}(\R^d)$ and $u$
is not identically zero.

And so the claim holds. This completes the proof of
excluding the soliton-like solution in the sense of Theorem
\ref{three}.

\subsection{Low to high frequency cascade}
Finally, we turn to exclude the low to high frequency cascade
solution.

We argue by contradiction. Assume that there exists a solution $u:
\R\times\R^d\to\C$ which is a low to high frequency cascade solution
in the sense of Theorem \ref{three}. Assume also that Theorem
\ref{regular} holds. In particular, there exists $\varepsilon>0$
such that
\begin{equation}\label{jiashe}
u(t,x)\in L_t^\infty(\R; \dot{H}_x^{-\varepsilon}(\R^d)).
\end{equation}

From $\varlimsup\limits_{t\to+\infty}N(t)=+\infty$, we can find a
time sequence $\{t_n\}$ such that
\begin{equation}\label{jiasmall}
\lim_{n\to+\infty}N(t_n)=+\infty.
\end{equation}

Using Bernstein's inequality, interpolation, the compactness
\eqref{xiaoc},  the hypothesis \eqref{jiashe}, and the assumption
\eqref{assume1.1}, we have
\begin{align}\nonumber
\|u(t_n,x)\|_{L_x^2}\leq&\|P_{\leq
c(\eta)N(t_n)}u\|_{L^2_x}+\|P_{\geq
c(\eta)N(t_n)}u\|_{L^2_x}\\\nonumber
\lesssim&\|u\|_{\dot{H}^{-\varepsilon}}^\frac{s_c}{s_c+\varepsilon}\big\|P_{\leq
c(\eta)N(t_n)}u\big\|_{\dot{H}^{s_c}_x}^{\frac{\varepsilon}{s_c+\varepsilon}}+\big(c(\eta)N(t_n)\big)^{-s_c}\|u\|_{\dot{H}^{s_c}_x}\\\label{jiashe1}
\lesssim&\eta^{\frac{\varepsilon}{s_c+\varepsilon}}+\big(c(\eta)N(t_n)\big)^{-s_c},
\end{align}
Taking $\eta$ small, and then $n$ large, we have by \eqref{jiasmall}
and \eqref{jiashe1}
$$M(u_0)=M(u(t_n)) \to 0,~\text{as}~n\to\infty,$$
which implies that $u\equiv0.$ This contradicts with the fact that
$u$ is a blowup solution. Therefore, we preclude the low to high
frequency cascade solution in the sense of Theorem \ref{three}.

In sum, it reduces to prove Theorem \ref{regular}.

\section{Negative regularity}
 \setcounter{section}{4}\setcounter{equation}{0}

As stated in Section 3, it remains to show Theorem \ref{regular}.
That is, we need to prove that the global solutions to
\eqref{equ1.1} which are almost periodic modulo symmetries enjoy the
negative regularity. We will divide two steps to prove it. First, we
show additional decay for the soliton-like and frequency-cascade
solutions in the sense of Theorem \ref{three}. And then, this
together with  the double Duhamel trick yields the negative
regularity for the soliton-like and frequency-cascade solutions.

\subsection{Additional Decay}
We first consider the energy-supercritical case.
\begin{proposition}[Additional decay I, energy-supercritical]\label{adecay1}  Let $d\geq9$ and $s_c>2.$
Assume in addition that $s_c<2+p$ if $p$ is not an even integer. And
let $u$ be a global solution to \eqref{equ1.1} that is almost
periodic modulo symmetries. In particular,
\begin{equation}\label{js}
\|u\|_{L_t^\infty(\R;\dot{H}^{s_c}_x(\R^d))}<+\infty.
\end{equation} And assume
that $\inf\limits_{t\in\R}N(t)\geq1.$ Then, we have
\begin{equation}\label{ad1}
u\in L_t^\infty(\R;L_x^q(\R^d)),\quad
q\in\Big(\frac{2d}{d-4},\frac{d}4p\Big].
\end{equation}
\end{proposition}

\begin{remark}\label{rem}
$(i)$ It is easy to see that we have by Sobolev embedding  and
\eqref{js}
$$u\in L_t^\infty  L_x^\frac{pd}{4}(\R\times\R^d).$$

$(ii)$ \eqref{ad1} can be reduced to show that there exists $\al>0$
and $N_0\in2^{\Z}$ such that for all dyadic number $N\leq N_0$
$$\|u_N\|_{L_t^\infty L_x^q(\R^d)}\lesssim N^{\al},\quad
q\in\Big(\frac{2d}{d-4},\frac{d}4p\Big].$$ In fact, we have by
Bernstein's inequality and \eqref{js}
\begin{align*}
\|u\|_{L_x^q(\R^d)}\lesssim&\sum_{N\leq N_0}\|u_N\|_{L_x^q}+\big\|P_{\geq N_0}u\big\|_{L_x^q}\\
\lesssim&\sum_{N\leq N_0}N^{\al}+\big\||\nabla|^{\frac{d}2-\frac{d}q}P_{\geq N_0}u\big\|_{L_x^2}\\
\lesssim&N_0^{\al}+N_0^{\frac{d}2-\frac{d}q-s_c}\big\||\nabla|^{s_c}u\big\|_{L_x^2}<+\infty.
\end{align*}
\end{remark}

\noindent{\bf The proof of Proposition \ref{adecay1}:} From
\eqref{xiaoc}, we know that $$\|u_{\leq
c(\eta)N(t)}\|_{\dot{H}^{s_c}}\leq\eta.$$ Combining this with
$\inf\limits_{t\in\R}N(t)\geq1$, we deduce that if we take $N_0$
such that $N_0\leq c(\eta)$, then
\begin{equation}\label{jxsmall}
\|u_{\leq N_0}\|_{\dot{H}^{s_c}}\leq\eta.
\end{equation}

Now we define $A_q(N)$ by
\begin{equation}\label{aqn1}
A_q(N)=N^{\frac{d}q-\frac4p}\|u_N\|_{L_t^\infty(\R;L^q_x(\R^d))},~q>\frac{2d}{d-4},
\end{equation}
It is easy to see that $A_q(N)\lesssim1$ by Bernstein's inequality
and \eqref{js}.

We first  consider that $p$ is an even integer.

{\bf Case 1: $p$ even.} We claim that  $A_q(N)$ satisfies the
following recurrence formula
\begin{equation}\label{diedai125}
\begin{split}
 A_q(N)
\lesssim&\Big(\frac{N}{N_0}\Big)^{d-4-\frac4p-\frac{d}q}+\eta^p\sum_{\frac{N}{10p}\leq
M\leq
N_0}\Big(\frac{N}{M}\Big)^{d-4-\frac4p-\frac{d}q}A_q(M)\\
&+\eta^p
\sum_{M\leq\frac{N}{10p}}\Big(\frac{M}{N}\Big)^{-\frac{d}2+4+\frac{d}{q}-}A_q(M)
\end{split}
\end{equation}
for any $q>\frac{2d}{d-4}$. Note that
$d-4-\frac4p-\frac{d}q,~-\frac{d}2+4+\frac{d}{q}->0$ whenever
$q\in\Big(\frac{2d}{d-4},\frac{2d}{d-8}\Big).$

We postpone the proof of this claim.  And we recall a acausal
Gronwall inequality.

\begin{lemma}[Acausal Gronwall inequality \cite{KV}]\label{gronwall1} Given
$\eta,C,\gamma,\gamma'>0$, let $\{x_k\}_{k\geq0}$ be a bounded
nonnegative sequence obeying
\begin{equation}\label{dd}
x_k\leq C2^{-\gamma
k}+\eta\sum_{l=0}^{k-1}2^{-\gamma(k-l)}x_l+\eta\sum_{l\geq
k}2^{-\gamma'(l-k)}x_l
\end{equation}
for all $k\geq0.$ If
$\eta\leq\frac14\min\{1-2^{-\gamma},1-2^{-\gamma'},1-2^{\rho-\gamma}\}$
for some $~0<\rho<\gamma,$ then
\begin{equation}
x_k\leq(4C+\|x\|_{l^\infty})2^{-\rho k}.
\end{equation}
\end{lemma}
Now we use the claim \eqref{diedai125} to prove Proposition
\ref{adecay1} for $p$ being an even integer.  Applying Lemma
\ref{gronwall1} with $x_k=A_q(2^{-k}N_0)$, we obtain by
\eqref{diedai125}
$$x_k\leq
C2^{-k\big(d-4-\frac4p-\frac{d}q\big)}+C\eta^p\sum_{l=0}^k2^{-(k-l)\big(d-4-\frac4p-\frac{d}q\big)}
x_l+C\eta^p\sum_{l>k}2^{-(l-k)\big(-\frac{d}2+4+\frac{d}{q}-\big)}x_l.$$
 Then
$x_k\lesssim 2^{-k\rho},~0<\rho<d-4-\frac2p-\frac{d}q,$ that is,
\begin{equation*}\label{add1}
A_q(N)\lesssim N^{(d-4-\frac4p-\frac{d}q)-},~q\in
\Big(\frac{2d}{d-4},\frac{2d}{d-8}\Big)\end{equation*}which means
for $N\leq N_0$
\begin{equation}\label{didaixh}
\|u_N\|_{L_x^q}\lesssim
N^{\big(d-4-\frac{2d}q\big)-},~q\in\Big(\frac{2d}{d-4},\frac{2d}{d-8}\Big).
\end{equation}
This together with Remark \ref{rem} (ii) yields that $$u\in
L_t^\infty(\R;L^q_x(\R^d)),~q\in\Big(\frac{2d}{d-4},\min\Big\{\frac{2d}{d-8},\frac{dp}4\Big\}\Big).$$
Interpolating this with $u\in L_t^\infty
L_x^\frac{pd}4(\R\times\R^d)$, we conclude Proposition \ref{adecay1}
for $p$ being an even integer.

Therefore, it suffices to prove the claim \eqref{diedai125}. By
time-translation symmetry, we only need to estimate \eqref{aqn1} at
$t=0$. Using No waste Duhamel formula \eqref{nwd}, Bernstein's
inequality, dispersive estimate \eqref{dispers}, we obtain for all
$q>\frac{2d}{d-4}$
\begin{align*}
\|u_n(0)\|_{L_x^q}\leq&\Big\|\int_0^{+\infty} e^{it\Delta^2}P_Nf(u)(t)dt\Big\|_{L_x^q}\\
\lesssim& N^{d\big(\frac12-\frac1q\big)}\int_0^{N^{-4}} \Big\|
e^{it\Delta^2}P_Nf(u)\Big\|_{L_x^2}dt+\int_{N^{-4}}^{+\infty}
t^{-d\big(\frac14-\frac1{2q}\big)}\|P_Nf(u)\|_
{L_x^{q'}}dt\\
\lesssim&N^{d-4-\frac{2d}q}\|P_Nf(u)\|_{L_t^\infty L_x^ {q'}}.
\end{align*}
Thus \begin{equation}\label{dd11} A_q(N)\lesssim
N^{d-4-\frac4p-\frac{d}q}\|P_Nf(u)\|_{L_t^\infty L_x^
{q'}},~q>\frac{2d}{d-4}. \end{equation} Decomposing $u$ by
$$u=u_{>N_0}+u_{\leq
N_0}=u_{>N_0}+u_{\frac{N}{10p}\leq\cdot\leq
N_0}+u_{<\frac{N}{10p}}$$ and using the fact that $p$ is an even
integer, we can write $P_Nf(u)$ by
\begin{equation}\label{fenjie}
P_Nf(u)=P_N\Big[\varnothing\Big(u_{>N_0}\sum_{k=0}^pu_{>N_0}^ku_{\leq
N_0}^{p-k}\Big)+\varnothing\Big(\sum_{k=0}^pu_{<\frac{N}{10p}}^ku_{\frac{N}{10p}\leq\cdot\leq
N_0}^{p+1-k}\Big)\Big].
\end{equation}
Here we use the notation $\varnothing(X)$ to denote a quantity that
resembles $X,$ that is, a finite linear combination of terms that
look like those in $X$, but possibly with some factors replaced by
their complex conjugates and/or restricted to various frequencies.

We first consider the terms which contain at least one factor of
$u_{>N_0}$. By H\"older's inequality,  Bernstein's inequality,
 Sobolev embedding: $
\dot{H}^{s_c}_x(\R^d)\hookrightarrow L^{\frac{pd}4}_x(\R^d)$ and the
assumption \eqref{js}, we get
\begin{align}\nonumber
\big\|P_N\varnothing\big(u_{>N_0}\cdot u^p\big)\big\|_{L_t^\infty
L_x^ {q'}}\lesssim&\|u_{> N_0}\|_{L_t^\infty L_x^r}\|u\|_{L_t^\infty
L_x^\frac{pd}4}^p\\\nonumber
\lesssim&N_0^{-d+4+\frac4p+\frac{d}q}\|u\|_{L_t^\infty
\dot{H}^{s_c}}^{p+1}\\\label{di1}
\lesssim&N_0^{-d+4+\frac4p+\frac{d}q},
\end{align}
where $1-\frac1q=\frac1r+\frac4d$.

To estimate the contribution of the second term on the right-hand
side of \eqref{fenjie} to \eqref{dd11}, we first note that
\begin{align*}
\big\|P_N\varnothing\Big(\sum_{k=0}^pu_{<\frac{N}{10p}}^ku_{\frac{N}{10p}\leq\cdot\leq
N_0}^{p+1-k}\Big)\big\|_{L_t^\infty
L_x^{q'}}\lesssim\big\|\varnothing\big(u_{\frac{N}{10p}\leq\cdot\leq
N_0}^{p+1}\big)\big\|_{L_t^\infty
L_x^{q'}}+\big\|\varnothing\big(u_{<\frac{N}{10p}}^pu_{\frac{N}{10p}\leq\cdot\leq
N_0}\big)\big\|_{L_t^\infty L_x^{q'}}.
\end{align*}
Using H\"older's inequality, Bernstein's inequality, the assumption
\eqref{js} and compactness \eqref{jxsmall}, we estimate
\begin{align}\nonumber
\|\varnothing\big(u_{\frac{N}{10p}\leq\cdot\leq
N_0}^{p+1}\big)\|_{L_t^\infty L_x^
{q'}}\lesssim&\|u_{\frac{N}{10p}\leq\cdot\leq N_0}\|_{L_t^\infty
L_x^ {\frac{d}4p}}^{p-1}\sum_{\frac{N}{10p}\leq M_1\leq M_2\leq
N_0}\|u_{M_1}\|_{L_t^\infty L_x^q}\|u_{M_2}\|_{L_t^\infty
L_x^r}\\\nonumber \lesssim&\eta^{p-1}\sum_{\frac{N}{10p}\leq M_1\leq
M_2\leq N_0}\|u_{M_1}\|_{L_t^\infty
L_x^q}M_2^{-d+4+\frac{2d}{q}}\|u_{\leq N_0}\|_{L_t^\infty
\dot{H}^{s_c}}\\\label{di2} \lesssim&\eta^p
N^{-d+4+\frac4p+\frac{d}q}\sum_{\frac{N}{10p}\leq M\leq
N_0}\Big(\frac{N}{M}\Big)^{d-4-\frac4p-\frac{d}q}A_q(M),
\end{align}
where $1-\frac1q=\frac{4(p-1)}{pd}+\frac1q+\frac1r,$ and we use the
fact $q>\frac{2d}{d-4}$ in the last inequality.

Similarly, we estimate
\begin{align*}
&\big\|\varnothing\big(u_{<\frac{N}{10p}}^pu_{\frac{N}{10p}\leq\cdot\leq
N_0}\big)\big\|_{L_t^\infty L_x^
{q'}}\\
\lesssim&\|u_{\frac{N}{10p}\leq\cdot\leq N_0}\|_{L_t^\infty
L_x^2}\sum_{M_1\leq\cdots\leq
M_p\leq\frac{N}{10p}}\prod_{j=1}^{p-1}\|u_{M_j}\|_{L_{t,x}^\infty}\|u_{M_p}\|_{L_t^\infty
L_x^\frac{2q}{q-2}}\\
\lesssim& \eta^2N^{-s_c}\sum_{M_1\leq\cdots\leq
M_p\leq\frac{N}{10p}}\prod_{j=1}^{p-1}M_j^\frac4pA_q(M_j)M_p^{-\frac{d}2+\frac{d}q+\frac4p}\\
\lesssim&\eta^2 N^{-s_c}\sum_{M_1\leq\cdots\leq
M_{p-1}\leq\frac{N}{10p}}M_1^{\varepsilon(p-1)}M_{p-1}^{-\frac{d}2+\frac{d}q+\frac4p}\\
&\times
\Big(M_1^{(\frac4p-\varepsilon)(p-1)}A_q(M_1)^{p-1}+M_{2}^{\frac{4(p-1)}{p}}A_q(M_{2})^{p-1}
+\cdots+M_{p-1}^{\frac{4(p-1)}{p}}A_q(M_{p-1})^{p-1}\Big)\\
\lesssim&\eta^2N^{-d+4+\frac4p+\frac{d}q}
\sum_{M\leq\frac{N}{10p}}\Big(\frac{M}{N}\Big)^{-\frac{d}2+4+\frac{d}q-}A_q(M)^{p-1}\\
\lesssim&\eta^pN^{-d+4+\frac4p+\frac{d}q}
\sum_{M\leq\frac{N}{10p}}\Big(\frac{M}{N}\Big)^{-\frac{d}2+4+\frac{d}q-}A_q(M),
\end{align*}
where $\varepsilon$ is a sufficiently small positive constant, and
we use $q>\frac{2d}{d-4}$ and $A_q(M)\lesssim\eta$ with $M\leq N_0$
in the above inequality. This together with \eqref{di1}, \eqref{di2}
and \eqref{dd11} imply the claim \eqref{diedai125}. And thus, we
conclude Proposition \ref{adecay1} for $p$ being an even integer.

{\bf Case 2: $p$ not even.} Now we turn to consider that $p$ is not
an even integer and $s_c\in(2,2+p).$

By the same argument as above, we estimate
\begin{equation}\label{dd1123} A_q(N)\lesssim
N^{d-4-\frac4p-\frac{d}q}\|P_Nf(u)\|_{L_t^\infty L_x^
{q'}},~q>\frac{2d}{d-4}. \end{equation} For $N\leq N_0$, using the
fundamental Theorem of Calculus, we decompose $f(u)$ by
\begin{align}
f(u)=&\varnothing\big(u_{>N_0}\cdot u_{\leq
N_0}^p\big)+\varnothing\big(u_{>N_0}^{p+1}\big)+f\big(u_{\frac{N}{10}\leq\cdot\leq
N_0}\big)\\\label{contru1}
&+u_{\leq\frac{N}{10}}\int_0^1f_z\big(u_{\frac{N}{10}\leq\cdot\leq
N_0}+\theta u_{<\frac{N}{10}}\big)d\theta\\\label{contru2}
&+\overline{u_{\leq\frac{N}{10}}}\int_0^1f_{\bar{z}}\big(u_{\frac{N}{10}\leq\cdot\leq
N_0}+\theta u_{<\frac{N}{10}}\big)d\theta.
\end{align}

The contribution to the right-hand side of \eqref{dd1123} coming
from that contain at least one copy of $u_{>N_0}$ can be estimated
by the same argument as \eqref{di1}.

 By a simple computation, we have the following
equivalence for $p$ being not an even integer
\begin{align*}
s_c<2+p~&\Longleftrightarrow~2p^2-(d-4)p+8>0,\\
\begin{cases}
\frac8{d-4}<p\leq1,\\
2<s_c<2+p
\end{cases}&\Longleftrightarrow
\begin{cases}
d=13,~p\leq1<p_1:=\frac{(d-4)-\sqrt{(d-4)^2-64}}4\\
d\geq14,~p<p_1
\end{cases}\\
\begin{cases}
p>\max\{1,\frac8{d-4}\},\\
2<s_c<2+p
\end{cases}
&\Longleftrightarrow~\begin{cases}
9\leq d\leq12:~p>\frac8{d-4},\\
d=13:~1<p<p_1~\text{or}~p>p_2,\\
d\geq14:~p>p_2:=\frac{(d-4)+\sqrt{(d-4)^2-64}}4.
\end{cases}
\end{align*}

Next, we divide two cases to estimate the contribution coming from
the remain terms.

 {\bf Subcase 2(i): $p\leq1$.} In this case we have
only $f_z(u)\in C^p(\C)$.

We first consider the contribution coming from the term
$f\big(u_{\frac{N}{10}\leq\cdot\leq N_0}\big)$. Using $l^p\subset
l^1$, H\"older's inequality, Bernstein's inequality and compactness
\eqref{jxsmall}, we deduce that
\begin{align*}
&\|f(u_{\frac{N}{10}\leq\cdot\leq N_0})\|_{L_t^\infty L_x^
{q'}}\\
\lesssim&\sum_{\frac{N}{10}\leq M_1\leq
N_0}\big\|u_{M_1}\big|u_{\frac{N}{10}\leq\cdot\leq
N_0}\big|^p\big\|_{L_t^\infty L_x^{q'}}\\
\lesssim&\sum_{\frac{N}{10}\leq M_1,M_2\leq
N_0}\big\|u_{M_1}\big|u_{M_2}\big|^p\big\|_{L_t^\infty L_x^{q'}}\\
\lesssim&\sum_{\frac{N}{10}\leq M_1\leq M_2\leq
N_0}\|u_{M_1}\|_{L_t^\infty L_x^q}\|u_{M_2}\|_{L_t^\infty
L_x^\frac{pq}{q-2}}^p\\
&+\sum_{\frac{N}{10}\leq M_2\leq M_1\leq N_0}\|u_{M_1}\|_{L_t^\infty
L_x^\frac{pq}{q-2}}^p\|u_{M_1}\|_{L_t^\infty
L_x^q}^{1-p}\|u_{M_2}\|_{L_t^\infty L_x^q}^p\\
\lesssim&\eta^p\sum_{\frac{N}{10}\leq M\leq
N_0}M^{-d+4+\frac4p+\frac{d}q}A_q(M)\\
&+\eta^p\sum_{\frac{N}{10}\leq M_2\leq M_1\leq
N_0}\Big(\frac{M_2}{M_1}\Big)^{2p(d-4-\frac4p-\frac{d}q)}\big(M_1^{-d+4+\frac4p+\frac{d}q}A_q(M_1)\big)^{1-p}
\big(M_2^{-d+4+\frac4p+\frac{d}q}A_q(M_2)\big)^{p}\\
\lesssim&\eta^pN^{-d+4+\frac4p+\frac{d}q}\sum_{\frac{N}{10}\leq
M\leq N_0}\Big(\frac{N}{M}\Big)^{d-4-\frac4p-\frac{d}q}A_q(M).
\end{align*}

Now we consider the contribution coming from \eqref{contru1} and
\eqref{contru2}. It suffices to consider \eqref{contru1}, since
similar arguments can be used to deal with \eqref{contru2}. By
H\"older's inequality, we obtain
\begin{align}\nonumber
&\Big\|P_N\Big(u_{\leq\frac{N}{10}}\int_0^1f_z\big(u_{\frac{N}{10}\leq\cdot\leq
N_0}+\theta u_{<\frac{N}{10}}\big)d\theta\Big)\Big\|_{L_t^\infty
L_x^ {q'}}\\\nonumber \lesssim&\|u_{<\frac{N}{10}}\|_{L_t^\infty
L_x^r}\Big\|P_N\big(\int_0^1f_z\big(u_{\frac{N}{10}\leq\cdot\leq
N_0}+\theta u_{<\frac{N}{10}}\big)d\theta\big)\Big\|_{L_t^\infty
L_x^ {\frac{d}{p+4}}}\\\label{last}
\lesssim&\|u_{<\frac{N}{10}}\|_{L_t^\infty
L_x^r}\big\|P_N\big(f_z\big(u_{\leq N_0}\big)\big)\big\|_{L_t^\infty
L_x^ {\frac{d}{p+4}}},
\end{align}
where $1-\frac1q=\frac1r+\frac{p+4}d.$ On the other hand, it follows
from the nonlinear Bernstein inequality \eqref{nlbernstein} that
$$\big\|P_N\big(f_z\big(u_{\leq
N_0}\big)\big)\big\|_{L_t^\infty L_x^ {\frac{d}{p+4}}}\lesssim
N^{-p}\big\|\nabla u_{\leq N_0}\big\|_{L_t^\infty L_x^
{\frac{pd}{p+4}}}^p\lesssim N^{-p}\big\||\nabla|^{s_c}u_{\leq
N_0}\big\|_{L_t^\infty L_x^2}^p.$$ Plugging this into \eqref{last},
and by Bernstein's inequality, compactness \eqref{jxsmall} we derive
\begin{align*}
&\Big\|P_N\Big(u_{\leq\frac{N}{10}}\int_0^1f_z\big(u_{\frac{N}{10}\leq\cdot\leq
N_0}+\theta u_{\leq N_0}\big)d\theta\Big)\Big\|_{L_t^\infty L_x^
{q'}}\\
\lesssim& N^{-p}\big\||\nabla|^{s_c}u_{\leq N_0}\big\|_{L_t^\infty
L_x^2}^p\|u_{<\frac{N}{10}}\|_{L_t^\infty L_x^r}\\
\lesssim&\eta^p
N^{-d+4+\frac4p+\frac{d}q}\sum_{M<\frac{N}{10}}\Big(\frac{M}{N}\Big)^{p+4+\frac{d}{q}+\frac4p-d}A_q(M).
\end{align*}

Putting everything together, we deduce that $A_q(N)$ satisfies the
following recurrence formula
\begin{equation}\label{diedai16}
\begin{split}
 A_q(N)
\lesssim&\Big(\frac{N}{N_0}\Big)^{d-4-\frac4p-\frac{d}q}+\eta^p\sum_{\frac{N}{10}\leq
M\leq
N_0}\Big(\frac{N}{M}\Big)^{d-4-\frac4p-\frac{d}q}A_q(M)\\
&+\eta^p
\sum_{M\leq\frac{N}{10}}\Big(\frac{M}{N}\Big)^{-d+4+p+\frac{d}{q}+\frac4p}A_q(M)
\end{split}
\end{equation}
for any $q\in \big(\frac{2d}{d-4},\frac{d}{d-4-p-\frac4p}\big)$.
Applying Lemma \ref{gronwall1} again, we obtain
\begin{equation}
A_q(N)\lesssim N^{(d-4-\frac4p-\frac{d}q)-},~q\in
\Big(\frac{2d}{d-4},\frac{d}{d-4-p-\frac4p}\Big)\end{equation}which
means for $N\leq N_0$
\begin{equation}\label{didaixh}
\|u_N\|_{L_x^q}\lesssim
N^{\big(d-4-\frac{2d}q\big)-},~q\in\Big(\frac{2d}{d-4},\frac{d}{d-4-p-\frac4p}\Big).
\end{equation}
This together with Remark \ref{rem} (ii) yields that $$u\in
L_t^\infty(\R;L^q_x(\R^d)),~q\in\Big(\frac{2d}{d-4},\min\Big\{\frac{d}{d-4-p-\frac4p},\frac{dp}4\Big\}\Big).$$
Interpolating this with $u\in L_t^\infty
L_x^\frac{pd}4(\R\times\R^d)$, we conclude Proposition \ref{adecay1}
for $p\in\big(\frac8{d-4},1\big]$.

{\bf Subcase 2(ii): $p>\max\{1,\frac8{d-4}\}$.}  By the same
argument as \eqref{di2}, we estimate
\begin{align*}
\|f\big(u_{\frac{N}{10p}\leq\cdot\leq N_0}\big)\|_{L_t^\infty L_x^
{q'}}\lesssim\eta^p N^{-d+4+\frac4p+\frac{d}q}\sum_{\frac{N}{10}\leq
M\leq N_0}\Big(\frac{N}{M}\Big)^{d-4-\frac4p-\frac{d}q}A_q(M).
\end{align*}

Next, we consider the contribution coming from \eqref{contru1} and
\eqref{contru2}. It suffices to consider \eqref{contru1}, since
similar arguments can be used to deal with \eqref{contru2}. Given
$p$, there exists $\varepsilon>0$ such that $s_c<2+p-\varepsilon.$
Using the H\"older, Bernstein's inequalities and compactness
\eqref{jxsmall}, we derive that
\begin{align}\label{diff}
&\Big\|P_N\Big(u_{\leq\frac{N}{10}}\int_0^1f_z\big(u_{\frac{N}{10}\leq\cdot\leq
N_0}+\theta u_{<\frac{N}{10}}\big)d\theta\Big)\Big\|_{L_t^\infty
L_x^ {q'}}\\\nonumber
 \lesssim&\|u_{<\frac{N}{10}}\|_{L_t^\infty
L_x^{r_1}}\Big\|P_{>\frac{N}{10}}\big(\int_0^1f_z\big(u_{\frac{N}{10}\leq\cdot\leq
N_0}+\theta u_{<\frac{N}{10}}\big)d\theta\big)\Big\|_{L_t^\infty
L_x^ {r_2}}\\\nonumber
\lesssim&N^{-s_c+2-\varepsilon}\|u_{<\frac{N}{10}}\|_{L_t^\infty
L_x^{r_1}}\big\||\nabla|^{s_c-2+\varepsilon} u_{\leq
N_0}\big\|_{L_t^\infty L_x^\frac{2d}{d-4+2\varepsilon}}\|u_{\leq
N_0}\|_{L_t^\infty L_x^\frac{pd}4}^{p-1}\\\nonumber \lesssim&\eta^p
N^{-d+4+\frac{d}q+\frac4p}\sum_{M\leq\frac{N}{10}}\Big(\frac{M}{N}\Big)^{-\frac{d}2+2+\varepsilon+\frac{d}q}A_q(M),
\end{align}
where $1-\frac1q=\frac1{r_1}+\frac1{r_2}$ and
$\frac1{r_2}=\frac{2d}{d-4+\varepsilon}+\frac{4(p-1)}{pd}.$ Thus, we
derive that $A_q(N)$ satisfies the following recurrence formula
\begin{equation}\label{diedai162}
\begin{split}
 A_q(N)
\lesssim&\Big(\frac{N}{N_0}\Big)^{d-4-\frac4p-\frac{d}q}+\eta^p\sum_{\frac{N}{10}\leq
M\leq
N_0}\Big(\frac{N}{M}\Big)^{d-4-\frac4p-\frac{d}q}A_q(M)\\
&+\eta^p
\sum_{M\leq\frac{N}{10}}\Big(\frac{M}{N}\Big)^{-\frac{d}2+2+\varepsilon+\frac{d}{q}}A_q(M)
\end{split}
\end{equation}
for any $q\in \big(\frac{2d}{d-4},\frac{2d}{d-4-2\varepsilon}\big)$.
Applying Lemma \ref{gronwall1} again, we get
\begin{equation}
A_q(N)\lesssim N^{(d-4-\frac4p-\frac{d}q)-},~q\in
\Big(\frac{2d}{d-4},\frac{2d}{d-4-2\varepsilon}\Big)\end{equation}which
means for $N\leq N_0$
\begin{equation}\label{didaixh}
\|u_N\|_{L_x^q}\lesssim
N^{\big(d-4-\frac{2d}q\big)-},~q\in\Big(\frac{2d}{d-4},\frac{2d}{d-4-2\varepsilon}\Big).
\end{equation}
This together with Remark \ref{rem} (ii) yields that $$u\in
L_t^\infty(\R;L^q_x(\R^d)),~q\in\Big(\frac{d}{d-4-\frac4p},\min\Big\{\frac{2d}{d-4-2\varepsilon},\frac{dp}4\Big\}\Big).$$
Interpolating this with $u\in L_t^\infty
L_x^\frac{pd}4(\R\times\R^d)$, we conclude Proposition \ref{adecay1}
for $p>\max\{1,\frac8{d-4}\}$.

Therefore, we complete the proof of Proposition \ref{adecay1}.

\vskip 0.2in The next result shows the additional decay for the
energy-subcritical and energy-critical cases.
\begin{proposition}[Additional decay II]\label{adecay2}  Let $d\geq9$ and $1\leq s_c\leq 2.$
 And let $u$ be
a global solution to \eqref{equ1.1} that is almost periodic modulo
symmetries. Assume also $\inf\limits_{t\in\R}N(t)\geq1.$ Then we
have
\begin{equation}\label{ad123}
u\in L_t^\infty(\R;L_x^q(\R^d)),\quad q\in \begin{cases}
\Big(r_1,\frac{pd}{4}\Big],~r_1=\frac{2d(\frac8p-d+6+s_c)}{d(\frac8p-d+6)+2s_c(d-5-\frac4p)},~\text{if}~p>1,\\
\Big(r_2,\frac{pd}{4}\Big],~r_2=\frac{2d(2p+\frac8p-d+4+s_c)}{d(2p+\frac8p-d+4)+2s_c(d-4-\frac4p-p)},~\text{if}~p\leq1.
\end{cases}
\end{equation}
In particular, if $s_c=2;$ that is: $p=\frac8{d-4},$  then
\begin{equation}\label{ad1232}
u\in L_t^\infty(\R;L_x^q(\R^d)),\quad q\in \begin{cases}
\Big(\frac{2d}{d-3},\frac{2d}{d-4}\Big],~\text{if}~p>1,~i.e.~d<12,\\
\Big(\frac{2(d+4)}{d},\frac{2d}{d-4}\Big],~\text{if}~p\leq1,~i.e.~d\geq12.
\end{cases}
\end{equation}
\end{proposition}
\begin{remark}
It is easy to check that
$$r_1,r_2<\frac{pd}4,~\text{whenever}~p>\frac8d.$$
\end{remark}
{\bf The proof of Proposition \ref{adecay2}:} Noting that
$\frac{2d}{d-4}\leq \frac{d}{d-4-\frac4p}$ in this case and by the
similar argument as Proposition \ref{adecay1}, we have
 for $N\leq N_0$
\begin{equation}\label{didaixh1}
\|u_N\|_{L_t^\infty L_x^q}\lesssim
N^{\big(d-4-\frac{2d}q\big)-},~q\in\begin{cases}
\Big(\frac{d}{d-4-\frac4p},\frac{d}{d-5-\frac4p}\Big),~\text{if}~p>1,\\
\Big(\frac{d}{d-4-\frac4p},\frac{d}{d-4-p-\frac4p}\Big),~\text{if}~p\leq1.
\end{cases}
\end{equation}
We remark that one estimates the term \eqref{diff} in the case $p>1$
by a different way. Term \eqref{diff} is estimated by
\begin{align*}
&\Big\|P_N\Big(u_{\leq\frac{N}{10}}\int_0^1f_z\big(u_{\frac{N}{10}\leq\cdot\leq
N_0}+\theta u_{<\frac{N}{10}}\big)d\theta\Big)\Big\|_{L_t^\infty
L_x^ {q'}}\\
\lesssim&\|u_{<\frac{N}{10}}\|_{L_t^\infty
L_x^r}\Big\|P_{>\frac{N}{10}}\big(\int_0^1f_z\big(u_{\frac{N}{10}\leq\cdot\leq
N_0}+\theta u_{<\frac{N}{10}}\big)d\theta\big)\Big\|_{L_t^\infty
L_x^ {\frac{d}{5}}}\\
\lesssim&N^{-1}\|u_{<\frac{N}{10}}\|_{L_t^\infty L_x^r}\big\|\nabla
u_{\leq N_0}\big\|_{L_t^\infty L_x^\frac{pd}{p+4}}\|u_{\leq
N_0}\|_{L_t^\infty L_x^\frac{pd}4}^{p-1}\\
\lesssim&\eta^p
N^{-d+4+\frac{d}q+\frac4p}\sum_{M\leq\frac{N}{10}}\Big(\frac{M}{N}\Big)^{-d+5+\frac{d}q+\frac4p}A_q(M).
\end{align*}
Therefore, we get \eqref{didaixh1}.

{\bf Case 1: $p>1$.}  We have by \eqref{didaixh1} with
$q=\frac{d}{d-5-\frac4p}-$
\begin{equation}\label{diwei}
\|u_N\|_{L_t^\infty L_x^{\frac{d}{d-5-\frac4p}-}}\lesssim N^{(\frac8p-d+6)-}.
\end{equation}
On the other hand, from the assumption \eqref{assume1.1}: $u\in
L_t^\infty\dot{H}^{s_c}_x(\R\times\R^d)$, we know that
$$\|u_N\|_{L_t^\infty L_x^2}\lesssim N^{-s_c}.$$
Interpolating this with \eqref{diwei}, we deduce that for $N\leq
N_0$
$$\|u_N\|_{L_t^\infty L_x^{r_1+}}\lesssim N^{0+}.$$
Hence we obtain $u\in L_t^\infty
L_x^{r_1+}(\R\times\R^d)$ by Remark
\ref{rem} (ii). Interpolating this with $u\in L_t^\infty
\dot{H}^{s_c}\subset L_t^\infty L_x^\frac{pd}{4}$ again, we derive
that
$$u\in L_t^\infty(\R;
L_x^q(\R^d)),~q\in\Big(r_1,\frac{pd}{4}\Big],~r_1=\frac{2d(\frac8p-d+6+s_c)}{d(\frac8p-d+6)+2s_c(d-5-\frac4p)}.$$

{\bf Case 2: $p\leq1$.} By \eqref{didaixh1}, we get
\begin{equation} \label{gaowei} \|u_N\|_{L_t^\infty
L_x^{\frac{d}{d-4-p-\frac4p}-}}\lesssim N^{(2p+\frac8p-d+4)}.
\end{equation}
On the other hand, from the assumption \eqref{assume1.1}: $u\in
L_t^\infty\dot{H}^{s_c}_x(\R\times\R^d)$, we know that
$$\|u_N\|_{L_t^\infty L_x^2}\lesssim N^{-s_c}.$$
Combining this with \eqref{gaowei}, we obtain  for $N\leq N_0$
$$\|u_N\|_{L_t^\infty L_x^{r_2+}}\lesssim
N^{0+}.$$ This implies $u\in L_t^\infty
L_x^{r_2+}$ by Remark \ref{rem} (ii).
Interpolating this with $u\in L_t^\infty L_x^\frac{pd}{4}$ concludes
the proof of this proposition.

\subsection{Negative regularity}

Now we utilize the double Duhamel trick to show Theorem
\ref{regular}. First, we drive a preliminary lemma.

\begin{lemma}\label{lregd1}
 Suppose that $u\in L_t^\infty(\R;\dot{H}^{s}_x(\R^d))$ for some $s\in[0, s_c]$. Assume also that there exists a positive constant
 $\al$ independent of $s$ such that
 \begin{equation}\label{lregd11}
 \big\||\nabla|^{s}P_Nu\big\|_{L_t^\infty
 L_x^2}\lesssim_{s}N^\al,~\forall~N\leq1.
 \end{equation}
 Then, for any $\beta\in[0,\al)$, we have $u\in L_t^\infty
 (\R;\dot{H}^{s-\beta}_x(\R^d))$.
 \end{lemma}
\begin{proof}
Using Bernstein's inequality and the assumption $u\in
L_t^\infty(\R;\dot{H}^{s}_x(\R^d))$, we have
\begin{align*}
\big\||\nabla|^{s-\beta}u\big\|_{L_t^\infty
L_x^2}\lesssim&\sum_{N\leq
1}\big\||\nabla|^{s-\beta}P_Nu\big\|_{L_t^\infty
L_x^2}+\big\||\nabla|^{s-\beta}P_{\geq1}u\big\|_{L_t^\infty L_x^2}\\
\lesssim&\sum_{N\leq 1}
N^{-\beta}\big\||\nabla|^{s}P_Nu\big\|_{L_t^\infty
L_x^2}+\big\||\nabla|^{s}u\big\|_{L_t^\infty L_x^2}\\
\lesssim&\sum_{N\leq 1} N^{\al-\beta}+1<+\infty.
\end{align*}
This completes the proof of Lemma \ref{lregd1}.
\end{proof}

\noindent{\bf The proof of Theorem \ref{regular}:} From Lemma
\ref{lregd1}, we know that the proof of Theorem \ref{regular} is
reduced to show that for any $s\in[0, s_c]$, there exists a positive
constant $\al$ independent of $s$ such that
\begin{equation}\label{guijie}
\big\||\nabla|^{s}u_N\big\|_{L_t^\infty(\R;L_x^2)}\lesssim
N^\al.\end{equation} Indeed, we first apply \eqref{guijie} with
$s=s_c$. Then we conclude that $u\in
L_t^\infty(\R;\dot{H}_x^{s_c-\al+})$ by Lemma \ref{lregd1}. And then
we apply \eqref{guijie} with $s=s_c-\al+$ and obtain $u\in
L_t^\infty(\R;\dot{H}_x^{s_c-2\al+})$. Iterating this procedure
finitely many times, we derive $u\in
L_t^\infty(\R;\dot{H}_x^{-\varepsilon})$ for any
$0<\varepsilon<\al$.

Hence it remains to prove the claim \eqref{guijie}. We divide two
cases to discuss. First, we consider the energy-supercritical case.

{\bf Case 1: $s_c>2$ (energy-supercritical).} Assume that $u\in
L_t^\infty(\R;\dot{H}^s_x(\R^d))$ for some $0\leq s\leq s_c$. It
follows from the additional decay (Proposition \ref{adecay1}) that
\begin{equation}
u\in L_t^\infty(\R;L_x^q(\R^d)),\quad
q\in\Big(\frac{2d}{d-4},\frac{d}4p\Big].
\end{equation}
And so, we obtain by \eqref{fxxqd}
\begin{equation}\label{lregd21}
\big\||\nabla|^sf(u)\big\|_{L_t^\infty
L_x^r}\lesssim\big\||\nabla|^su\big\|_{L_t^\infty
L_x^2}\|u\|_{L_t^\infty L_x^q}^p\lesssim1,\quad
r=\frac{2q}{q+2p}<2.\end{equation} Then the condition $r\geq1$
requires $q\geq2p$.

It follows from No waste Duhamel formula \eqref{nwd} that
\begin{align*}
u(0)=&-\int_0^\infty e^{it\Delta^2}f(u)(t)dt=\int_{-\infty}^0
e^{i\tau\Delta^2}f(u)(\tau)d\tau.
\end{align*}
And so
\begin{align}\nonumber
\big\||\nabla|^su_N(0)\big\|_2^2 =&-\Big\langle\int_0^\infty
e^{it\Delta^2}|\nabla|^sP_Nf(u)(t)dt,\int_{-\infty}^0
e^{i\tau\Delta^2} |\nabla|^sP_Nf(u)(\tau)d\tau\Big\rangle\\\nonumber
=&-\int_0^\infty\int_{-\infty}^0\Big\langle
e^{i(t-\tau)\Delta^2}|\nabla|^sP_Nf(u)(t),
|\nabla|^sP_Nf(u)(\tau)\Big\rangle d\tau dt\\\label{jidi1}
\triangleq&\int_0^\infty\int_{-\infty}^0 F(t,\tau)d\tau dt,
\end{align}
where
$$F(t,\tau)=-\Big\langle e^{i(t-\tau)\Delta^2}|\nabla|^sP_Nf(u)(t),
|\nabla|^sP_Nf(u)(\tau)\Big\rangle.$$ On one hand, using the H\"older,
Bernstein inequalities and  \eqref{lregd21}, we get
\begin{align}\nonumber
F(t,\tau)\leq&\big\|e^{i(t-\tau)\Delta^2}|\nabla|^{s}P_Nf(u)\big\|_{L_x^2}\big\||\nabla|^sP_Nf(u)\big\|_{L_x^2}\\\nonumber
\lesssim&N^{2d\big(\frac1r-\frac12\big)}\big\||\nabla|^sP_Nf(u)\big\|_{L_x^r}^2\\\label{jidi2}
\lesssim&N^{2d\big(\frac1r-\frac12\big)}.
\end{align}
On the other hand, by H\"older's inequality and dispersive estimate \eqref{dispers}, we derive
\begin{align}\nonumber
F(t,\tau)\leq&\big\|e^{i(t-\tau)\Delta^2}|\nabla|^{s}P_Nf(u)\big\|_{L_x^{r'}}\big\||\nabla|^sP_Nf(u)\big\|_{L_x^r}\\\nonumber
\lesssim&(t-\tau)^{-d\big(\frac14-\frac1{2r'}\big)}\big\|||\nabla|^{s}P_Nf(u)\big\|_{L_x^r}^2\\\label{jidi3}
\lesssim&(t-\tau)^{-d\big(\frac1{2r}-\frac14\big)}.
\end{align}
Hence, plugging \eqref{jidi2} and \eqref{jidi3} into \eqref{jidi1},
we obtain
\begin{align}\nonumber
\big\||\nabla|^su_N(0)\big\|_2^2
\lesssim&\int_0^\infty\int_{-\infty}^0
\min\Big\{N^{2d\big(\frac1r-\frac12\big)},(t-\tau)^{-d\big(\frac1{2r}-\frac14\big)}\Big\}d\tau
dt\\\nonumber
\lesssim&\frac14\int_{\R}\int_{\R}\min\Big\{N^{2d\big(\frac1r-\frac12\big)},(|t|+|\tau|)^{-d\big(\frac1{2r}-\frac14\big)}\Big\}d\tau
dt\\\label{zhib} \lesssim&N^{\frac{2pd}q}\iint_{|t|+|\tau|\leq
N^{-4}}d\tau dt+\iint_{|t|+|\tau|\geq N^{-4}}
(|t|+|\tau|)^{-\frac{pd}{2q}} dtd\tau\\\nonumber \lesssim&
N^{-8+\frac{2pd}q},
\end{align}
where  $r=\frac{2q}{q+2p}$ and we also need the restriction
$\frac{pd}{2q}>2$ to guarantee the above integral converges.
Therefore,
\begin{equation}\label{jidi}
\big\||\nabla|^{s}u_N(0)\big\|_2\lesssim N^{-4+\frac{pd}q},\quad
q\in
\Big[2p,\frac{pd}{4}\Big)\bigcap\Big(\frac{2d}{d-4},\frac{pd}{4}\Big].
\end{equation}
The condition $2p<\frac{pd}{4}$ requires the dimension $d$ such that
$d\geq9.$ Now if we take $q=\max\big\{2p,\frac{2d}{d-4}+\big\}$,
then we obtain
$$\al=-4+\frac{pd}q=\min\Big\{-4+\frac{d}2,-4+\frac{(d-4)p}2-\Big\}>0.$$ Therefore we conclude
\eqref{guijie}. And so we complete the proof of Theorem
\ref{regular} for $s_c>2$.

We remark that the balance between the bounds provided by Lemma
\ref{adecay1} and the bound required by Theorem \ref{regular} is the
source of our restriction to dimensions $d\geq9.$ As we noted above,
\eqref{lregd21} provides the $L_t^\infty L_x^q$ bounds for
$q\geq2p,$ while  \eqref{zhib} requires this bound with
$q<\frac{pd}4$. These conditions on $q$ impose the restriction
$d\geq9.$

{\bf Case 2: $1\leq s_c\leq2$ (energy-subcritical and
energy-critical).} By the same argument as the energy-supercritical
case, we have
\begin{equation*}
\big\||\nabla|^{s}u_N(0)\big\|_2\lesssim N^{-4+\frac{pd}{q}},
\end{equation*}
where $q$ satisfies
\begin{equation}
q\in \Big[2p,\frac{pd}{4}\Big)\bigcap\begin{cases}
\Big(r_1,\frac{pd}{4}\Big],~\text{if}~p>1,\\
\Big(r_2,\frac{pd}{4}\Big],~\text{if}~p\leq1.
\end{cases}
\end{equation}
where $(r_1,r_2)$ is as in Proposition \ref{adecay2}. If we take $q=\frac{pd}{4}-,$ then we obtain
$\al=-4+\frac{pd}{q}=0+>0.$ Hence we get \eqref{guijie}. Therefore,
we complete the proof of Theorem \ref{regular} for $s_c\in[1,2]$.
Therefore, we conclude Theorem \ref{theorem}.

\section{Appendix}
In this appendix, we show the perturbation theory. We first consider
that $p$ is an even integer.
\subsection{Perturbation I: $p$ even} Here we give  a perturbation result under the weakest assumption on the
difference of the initial data  \eqref{eq2.222}, since it is easy to
show the smallness assumption \eqref{eq2.222} can be derived from
\eqref{eq2.2221} by Strichartz estimate.
\begin{lemma}[Perturbation Lemma, $p$ even]\label{pertu1}  Assume that $s_c=\frac{d}2-\frac4p\geq1,$ and $p$ is an even integer. Let $I$ be
a compact time interval and $u,~\tilde{u}$ satisfy
\begin{align*}
(i\partial_t+\Delta^2)u=&-f(u)+eq(u)\\
(i\partial_t+\Delta^2)\tilde{u}=&-f(\tilde{u})+eq(\tilde{u})
\end{align*}
for some function $eq(u),eq(\tilde{u})$, and $f(u)=|u|^pu$. Assume
that for some constants $M,E>0$, we have
\begin{align}\label{eq2.20}
\|u\|_{L_t^\infty(I;
\dot{H}^{s_c}_x(\R^d))}+\|\tilde{u}\|_{L_t^\infty(I;
\dot{H}^{s_c}_x(\R^d))}\leq E,\\\label{equ2.201} S_I(\tilde{u})\leq
M,
\end{align}
 Let $t_0\in I$, and let $u(t_0)$ be
close to $\tilde{u}(t_0)$ in the sense that
\begin{equation} \label{eq2.222}
\big\||\nabla|^{s_c-1}e^{i(t-t_0)\Delta^2}(u-\tilde{u})(t_0)\big\|_{L_t^{2(p+1)}L_x^\frac{2d(p+1)}{(d-2)(p+1)-4}(I\times\R^d)}\leq\varepsilon,
\end{equation}
where $0<\varepsilon<\varepsilon_1=\varepsilon_1( M, E)$ is a small
constant. Assume also that we have smallness conditions
\begin{equation}\label{equ2.21}
\big\||\nabla|^{s_c-1}\big(eq(u),eq(\tilde{u})\big)\big\|_{L_t^2L_x^\frac{2d}{d+2}(I\times\R^d)}\leq\varepsilon,
\end{equation}
where $\varepsilon$ is as above.

 Then
we conclude that
\begin{equation}\label{eq2.232}
\begin{split}
S_I(u-\tilde{u})\leq & C(M,E)\varepsilon\\
\big\||\nabla|^{s_c}(u-\tilde{u})\big\|_{S^0(I)}\leq &
C(M,E)\varepsilon\\
\big\||\nabla|^{s_c}u\big\|_{S^0(I)}\leq & C(M,E).
\end{split}
\end{equation}
\end{lemma}

\begin{proof}
Since $S_I(\tilde{u})\leq M$, we may subdivide $I$ into $C(M,
\varepsilon_0)$ time intervals $I_j$ such that
$$S_{I_j}(\tilde{u})\leq \epsilon_0\ll1, \quad \quad 1\le j\le C(M,\varepsilon_0). $$
 By the Strichartz estimate and standard bootstrap argument we
have
\begin{equation*}\label{}
\big\||\nabla|^{s_c}\tilde{u}\big\|_{S^0(I_j)}\leq C(E), \qquad 1\le
j\le C(M,\varepsilon_0).
\end{equation*}
Summing up over all the intervals, we obtain that
\begin{equation}\label{eq2.24}
\big\||\nabla|^{s_c}\tilde{u}\big\|_{S^0(I)}\leq C(E, M).
\end{equation}
In particular, we have by Sobolev embedding
\begin{equation}\label{eq2.25}
\|\tilde{u}\|_{Z(I)}:=\big\||\nabla|^{s_c-1}\tilde{u}\big\|_{L_t^{2(p+1)}L_x^\frac{2d(p+1)}{(d-2)(p+1)-4}(I\times\R^d)}\leq
C(E, M),
\end{equation}
which implies that there exists a partition of the right half of $I$
at $t_0$:
$$t_0<t_1<\cdots<t_N,~I_j=(t_j,t_{j+1}),~I\cap(t_0,\infty)=(t_0,t_N),$$
such that $N\leq C(L,\delta)$ and for any $j=0,1,\cdots,N-1,$ we
have
\begin{equation}\label{ome}
\|\tilde{u}\|_{Z(I_j)}\leq\delta\ll1.
\end{equation}
The estimate on the left half of $I$ at $t_0$ is analogue, we omit
it.

Let
\begin{equation}
\gamma(t)=u(t)-\tilde{u}(t),
\end{equation}
and \begin{equation} \gamma_j(t)
=e^{i(t-t_j)\Delta^2}\gamma(t_j),~0\leq j\leq N-1,
\end{equation} then $\gamma$ satisfies the following
difference equation
\begin{align*}
\begin{cases}
(i\partial_t+\Delta^2){\gamma}=-f(\tilde{u}+\gamma)+f(\tilde{u})+eq(u)-eq(\tilde{u}),\\
{\gamma}(t_j)={\gamma}_j(t_j),
\end{cases}
\end{align*}
which implies that
\begin{align*}
\gamma(t)=&\gamma_j(t)-i\int_{t_j}^t e^{i(t-s)\Delta^2}\big(-f(\tilde{u}+\gamma)+f(\tilde{u})+eq(u)-eq(\tilde{u})\big)ds,\\
\gamma_{j+1}(t)
=&\gamma_j(t)-i\int_{t_j}^{t_{j+1}}e^{i(t-s)\Delta^2}\big(-f(\tilde{u}+\gamma)+f(\tilde{u})+eq(u)-eq(\tilde{u})\big)ds.
\end{align*}
It follows from Strichartz estimate  and nonlinear estimate
\eqref{fxxqd} that
\begin{align}\label{equ3}
&\|\gamma-\gamma_j\|_{Z(I_j)}+\|\gamma_{j+1}-\gamma_j\|_{Z(I)}\\\nonumber
\lesssim&\big\||\nabla|^{s_c-1}\big(f(\tilde{u}-\gamma)+f(\tilde{u})\big)\big\|_{L_t^2L_x^\frac{2d}{d+2}(I\times\R^d)}
+\big\||\nabla|^{s_c-1}\big(eq(u),eq(\tilde{u})\big)\big\|_{L_t^2L_x^\frac{2d}{d+2}(I\times\R^d)}\\\nonumber
\lesssim&\sum_{k=1}^{p+1}\|\gamma\|_{Z(I_j)}^k\|\tilde{u}\|_{Z(I_j)}^{p+1-k}
+\big\||\nabla|^{s_c-1}\big(eq(u),eq(\tilde{u})\big)\big\|_{L_t^2L_x^\frac{2d}{d+2}(I\times\R^d)}\\\nonumber
\lesssim&\sum_{k=1}^{p+1}\|\gamma\|_{Z(I_j)}^k\|\tilde{u}\|_{Z(I_j)}^{p+1-k}+\varepsilon.
\end{align}
Therefore, assuming that
\begin{equation}\label{laodong}
\|\gamma\|_{Z(I_j)}\leq\delta\ll1,~\forall~j=0,1,\cdots,N-1,
\end{equation}
then by \eqref{ome} and \eqref{equ3}, we have
\begin{equation}
\|\gamma\|_{Z(I_j)}+\|\gamma_{j+1}\|_{Z(t_{j+1},t_N)}\leq
C\|\gamma_j\|_{Z(t_j,t_N)}+\varepsilon,
\end{equation}
for some absolute constant $C>0$. By \eqref{eq2.222} and iteration
on $j$, we obtain
\begin{equation}
\|\gamma\|_{Z(I)}\leq (2C)^N\varepsilon\leq\frac{\delta}2,
\end{equation}
if we choose $\varepsilon_1$ sufficiently small. Hence the
assumption \eqref{laodong} is justified by continuity in $t$ and
induction on $j$. Then repeating the estimate \eqref{equ3} once
again, we can get the critical-norm estimate on $\gamma$, which
implies the Strichartz estimates on $u$. This concludes the proof of
this lemma.
\end{proof}

\subsection{Perturbation II: $p$ not even}
In this subsection, we will establish  the perturbation theory of
the solution of \eqref{equ1.1} with $p$ being not an even integer.
We restate the perturbation lemma as follows.
\begin{lemma}[Perturbation Lemma, $p$ not even]\label{longpert1}  Assume that $p$ is not an
even integer and $1\leq s_c<2+p.$ Let $I$ be a compact time interval
and $u,~\tilde{u}$ satisfy
\begin{align*}
(i\partial_t+\Delta^2)u=&-f(u)+eq(u)\\
(i\partial_t+\Delta^2)\tilde{u}=&-f(\tilde{u})+eq(\tilde{u})
\end{align*}
for some function $eq(u),eq(\tilde{u})$, and $f(u)=|u|^pu$. Assume
that for some constants $M,E>0$, we have
\begin{align}\label{eq2.2230}
\|u\|_{L_t^\infty(I;
\dot{H}^{s_c}_x(\R^d))}+\|\tilde{u}\|_{L_t^\infty(I;
\dot{H}^{s_c}_x(\R^d))}\leq E,\\\label{equ2.20321}
S_I(\tilde{u})\leq M,
\end{align}
 Let $t_0\in I$, and let $u(t_0)$ be
close to $\tilde{u}$ in the sense that
\begin{equation} \label{eq2.22212}
\big\|u(t_0)-\tilde{u}(t_0)\big\|_{\dot{H}^{s_c}}\leq\varepsilon,
\end{equation}
where $0<\varepsilon<\varepsilon_1( M, E)$ is a small constant.
Assume also that we have smallness conditions
\begin{equation}\label{equ.1}
\big\||\nabla|^{s_c-1}\big(eq(u),eq(\tilde{u})\big)\big\|_{L_t^2L_x^\frac{2d}{d+2}(I\times\R^d)}\leq\varepsilon,
\end{equation}
where $\varepsilon$ is as above.

 Then
we conclude that
\begin{equation}\label{eq2.232423}
\begin{split}
S_I(u-\tilde{u})\leq & C(M,E)\varepsilon^{c_1}\\
\big\||\nabla|^{s_c}(u-\tilde{u})\big\|_{S^0(I)}\leq &
C(M,E)\varepsilon^{c_2}\\
\big\||\nabla|^{s_c}u\big\|_{S^0(I)}\leq & C(M,E),
\end{split}
\end{equation}
where $c_1,~c_2$ are positive constants that depend on $d,~p,~E$ and
$M$.
\end{lemma}

The proof of the above lemma with $p>s_c-1$ is similar to Lemma
\ref{pertu1} based on the use of the standard Strichartz estimates.
However this proof can not be applied directly to $p\leq s_c-1$. The
main reason for this is that for $p\leq s_c-1$ the derivative of the
nonlinearity is no longer Lipschitz continuous in the standard
Strichartz space. In \cite{TV}, Tao and Visan first overcame this
problem in the context of the energy-critical NLS in dimensions
$d>6$ by making use of certain ``exotic Strichartz" spaces which
have same scaling with standard Strichartz space but lower
derivative. Later, Killip and Visan simplified the proof in
\cite{KVnote} where stability is established in Sobolev Strichartz
spaces where they utilized the fractional chain rule.

Therefore, we always assume that $p\leq s_c-1$. We give a sketch
proof by the similar argument as in \cite{KV2010}. First, it is
useful to define several spaces and give estimates of the
nonlinearities in terms of these spaces. Given $s:=\frac{p}2$,
define
\begin{equation}\label{zhibxqyz}
\begin{split}
\|u\|_{X^0(I)}:=&\|u\|_{L_t^{q_0}L_x^r(I\times\R^d)}\triangleq
\|u\|_{L_t^{q_0}L_x^\frac{r_0d}{d-r_0s_c}(I\times\R^d)}=\|u\|_{L_t^{q_0}L_x^\frac{(p+2)d}{d-p}(I\times\R^d)}\\
\|u\|_{X(I)}:=&\big\||\nabla|^su\big\|_{L_t^{q_0}L_x^{r_1}(I\times\R^d)},\quad r_1=\frac{2r_0d}{2d-r_0(2s_c-p)}\\
\|F\|_{Y(I)}:=&\big\||\nabla|^sF\big\|_{L_t^\frac{q_0}{1+p}L_x^{r_1'}(I\times\R^d)},\quad
\frac1{r_1}+\frac1{r_1'}=1,
\end{split}
\end{equation}
where
$(q_0,r_0)=\Big(\frac{4p(p+2)}{p^2-p(d-4)+8},\frac{d(p+2)}{d-p+s_c(p+2)}\Big)$,
$2<r_0<\frac{d}{s_c},~\frac{d+4}4p<q_0<+\infty.$ It is easy to check
that $(q_0,r,r_1,s)$ satisfies
\begin{enumerate}
\item $(q_0,r)$:  $s_c$-admissible pair, that is
$\frac4{q_0}=d\Big(\frac12-\frac1r\Big)-s_c=\frac4p-\frac{d}r.$
\item $(q_0,r_1)$:  $(s_c-s)$-admissible pair, that is
$\frac4{q_0}=d\Big(\frac12-\frac1{r_1}\Big)-(s_c-s).$
\item Nonlinear estimate
\begin{equation}\label{nlesfs1}
\|f(u)\|_{Y(I)}\lesssim\|u\|_{X(I)}\|u\|_{X^0(I)}^p\lesssim\|u\|_{X(I)}^{p+1}
\end{equation}
requires $\frac1{r_1'}=\frac1{r_1}+\frac{p}{r}.$
\item ``Exotic Strichartz estimate"  Hardy-Littlewood-Sobolev
requires:
$$1+\frac1{q_0}=d\Big(\frac14-\frac1{2r_1}\Big)+\frac{p+1}{q_0}.$$
\end{enumerate}
It is easy to verify that the Sobolev embedding relations
\begin{align}\label{inter1}
\|u\|_{X^0(I)}\lesssim&\|u\|_{X(I)}\lesssim\big\||\nabla|^{s_c}u\big\|_{S^0(I)}
\end{align}
and interpolation implies that there exist  $0<\theta_1,~\theta_2<1$
such that
\begin{align}\label{inter2}
\|u\|_{X(I)}\lesssim&\|u\|_{L_{t,x}^\frac{(d+4)p}4(I\times\R^d)}^{\theta_1}\big\||\nabla|^{s_c}u\big\|_{S^0(I)}^{1-\theta_1}\\\label{inter3}
\|u\|_{L_{t,x}^\frac{(d+4)p}4(I\times\R^d)}\lesssim&\|u\|_{X(I)}^{\theta_2}\big\||\nabla|^{s_c}u\big\|_{S^0(I)}^{1-\theta_2},
\end{align}
Also, as a direct consequence of Hardy-Littlewood-Sobolev
inequality, we have the following ``exotic Strichartz estimate".
\begin{lemma}[Exotic Strichartz estimate]\label{exto}  Let $I$
be a compact time interval containing $t_0$ , then
\begin{equation}\label{exto1}
\Big\|\int_{t_0}^te^{i(t-s)\Delta^2}F(s)ds\Big\|_{X(I)}\lesssim\big\|F\big\|_{Y(I)}.
\end{equation}
\end{lemma}

\begin{proof}
It follows from the dispersive estimate \eqref{dispers} that
$$\big\|e^{i(t-s)\Delta^2}F(s)\big\|_{L_x^{r_1}}\lesssim|t-s|^{-\frac{p(d-r_0s_c)}{4r_0}}\big\|F(s)\big\|_{L_x^{r_1'}}.$$
This together with Hardy-Littlewood-Sobolev inequality yields that
\begin{align*}
\Big\|\int_{t_0}^te^{i(t-s)\Delta^2}F(s)ds\Big\|_{L_t^{q_0}L_x^{r_1}(I\times\R^d)}\lesssim&\Big\|\int_{t_0}^t\big\|e^{i(t-s)\Delta^2}F(s)
\big\|_{L_x^{r_1}}ds\Big\|_{L_t^{q_0}(I)}\\
\lesssim&\Big\|\int_{t_0}^t|t-s|^{-\frac{p(d-r_0s_c)}{4r_0}}\big\|F(s)\big\|_{L_x^{r_1'}}ds\Big\|_{L_t^{q_0}(I)}\\
\lesssim&\big\|F\big\|_{L_t^\frac{q_0}{p+1}L_x^{r_1'}(I\times\R^d)}.
\end{align*}
\end{proof}

\begin{lemma} [Nonlinear estimates]\label{nlefs} Let $d\geq9,~1\leq s_c<2+p$, and $I$ be a time interval. Then
\begin{align}\nonumber
\big\|f_z(u+v)\omega\big\|_{Y(I)}+\big\|f_{\bar{z}}(u+v)\bar{\omega}\big\|_{Y(I)}&\\\label{nlesf2}
\lesssim\Big(\|u\|_{X(I)}^\frac{p(s_c-1)}{s_c}\big\||\nabla|^{s_c}u\big\|_{S^0(I)}^\frac{p}{s_c}+&\|v\|_{X(I)}^\frac{p(s_c-1)}{s_c}
\big\||\nabla|^{s_c}v\big\|_{S^0(I)}^\frac{p}{s_c}\Big)\|\omega\|_{X(I)},
\end{align}
and there exists $\beta\in(0,p)$ such that
\begin{align}\label{nlesf11}
&\big\||\nabla|^{s_c-1}\big[f(u+v)-f(u)\big]\big\|_{L_t^2L_x^\frac{2d}{d+2}(I\times\R^d)}\\\nonumber
\lesssim&\big\||\nabla|^{s_c}u\big\|_{S^0(I)}\Big[\|v\|_{X^0(I)}^p+\|u\|_{X^0(I)}^{p-1+\frac1{s_c}}\|v\|_{X^0(I)}^{1-\frac1{s_c}}+\big(
\|u\|_{X^0(I)}^{p-1+\frac1{s_c}}+\|v\|_{X^0(I)}^{p-1+\frac1{s_c}}\big)\big\||\nabla|^{s_c}u\big\|_{S^0(I)}^{1-\frac1{s_c}}\Big]\\\nonumber
&+\big\||\nabla|^{s_c}u\big\|_{S^0(I)}\Big(\|v\|_{X^0(I)}^{p}+\|u\|_{X^0(I)}^{\beta}\|v\|_{X^0(I)}^{p-\beta}\Big)
+\big\||\nabla|^{s_c}u\big\|_{S^0(I)}^\frac1{s_c}\big\||\nabla|^{s_c}v\big\|_{S^0(I)}^{1-\frac1{s_c}}\|u\|_{X^0(I)}^{1-\frac1{s_c}}
\|v\|_{X^0(I)}^{p-1+\frac1{s_c}}.
\end{align}
\end{lemma}

\vskip0.2cm

\begin{proof}

{\bf The proof of \eqref{nlesf2}:} It suffices to prove the first
term on the left-hand side, as the second term can be estimate by
the same way. Using \eqref{moser} and \eqref{inter1}, we derive
\begin{align}\nonumber
\big\|f_z(u+v)\omega\big\|_{Y(I)}\lesssim&\big\|f_z(u+v)\big\|_{L_t^\frac{q_0}{p}L_x^\frac{r_0d}{p(d-r_0s_c)}}\|\omega\|_{X(I)}\\\nonumber
&+\big\||\nabla|^\frac{p}2f_z(u+v)\big\|_{L_t^\frac{q_0}pL_x^\frac{2r_0d}{p(2d-2r_0s_c+r_0)}}\|\omega\|_{X^0(I)}\\\label{nlesclaim}
\lesssim&\Big(\|u+v\|_{X^0(I)}^p+\big\||\nabla|^\frac{p}2f_z(u+v)\big\|_{L_t^\frac{q_0}pL_x^\frac{2r_0d}{p(2d-2r_0s_c+r_0)}}\Big)\|\omega\|_{X(I)}.
\end{align}
Hence, from \eqref{inter1}, we know that the proof of \eqref{nlesf2}
can be reduced to prove
\begin{align*}
&\big\||\nabla|^\frac{p}2f_z(u+v)\big\|_{L_t^\frac{q_0}pL_x^\frac{2r_0d}{p(2d-2r_0s_c+r_0)}}\\
\lesssim&\|u\|_{X(I)}^\frac{p(s_c-1)}{s_c}\big\||\nabla|^{s_c}u\big\|_{S^0(I)}^\frac{p}{s_c}+\|v\|_{X(I)}^\frac{p(s_c-1)}{s_c}
\big\||\nabla|^{s_c}v\big\|_{S^0(I)}^\frac{p}{s_c}.
\end{align*}

{\bf Case 1: $p\geq1$.} Using \eqref{fxxqd}  and \eqref{inter1}, we
get
$$\big\||\nabla|^\frac{p}2f_z(u+v)\big\|_{L_t^\frac{q_0}pL_x^\frac{2r_0d}{p(2d-2r_0s_c+r_0)}}\lesssim\|u+v\|_{X^0(I)}^{p-1}\|u+v\|_{X(I)}
\lesssim\|u+v\|_{X(I)}^p.$$

{\bf Case 2: $p<1$.} By the fractional chain rule \eqref{hfscls}
with $s=\frac{p}2,~\frac12<\sigma<1$, H\"older inequality, Sobolev
embedding and interpolation, we obtain
\begin{align*}
\big\||\nabla|^\frac{p}2f_z(u+v)\big\|_{L_t^\frac{q_0}pL_x^\frac{2r_0d}{p(2d-2r_0s_c+r_0)}}\lesssim&\|u+v\|_{X^0(I)}^{p-\frac{p}{2\sigma}}
\big\||\nabla|^\sigma(u+v)\big\|_{L_t^{q_0}L_x^\frac{r_0d}{d-r_0(s_c-\sigma)}}^\frac{p}{2\sigma}\\
\lesssim&\big\||\nabla|^\sigma(u+v)\big\|_{L_t^{q_0}L_x^\frac{r_0d}{d-r_0(s_c-\sigma)}}^{p}\\
\lesssim&\|u\|_{X(I)}^\frac{p(s_c-1)}{s_c}\big\||\nabla|^{s_c}u\big\|_{S^0(I)}^\frac{p}{s_c}+\|v\|_{X(I)}^\frac{p(s_c-1)}{s_c}
\big\||\nabla|^{s_c}v\big\|_{S^0(I)}^\frac{p}{s_c}.
\end{align*}
Plugging this into \eqref{nlesclaim}, we get \eqref{nlesf2}.

\noindent{\bf \quad The proof of \eqref{nlesf11}:} When
$s_c\in[1,2),$ it is easy to show \eqref{nlesf11} by Lemma
\ref{deridiff}, H\"older's inequality and Sobolev embedding.  Now we
consider $s_c\in[2,p+2).$ For $p\leq1.$ Using the Fundamental
Theorem of Calculus and triangle inequality, we deduce that
\begin{align}\nonumber
&\big\||\nabla|^{s_c-1}\big[f(u+v)-f(u)\big]\big\|_{L_t^2L_x^\frac{2d}{d+2}(I\times\R^d)}\\\nonumber
\lesssim&\big\||\nabla|^{s_c-2}\big[\nabla v\cdot
f'(u+v)\big]\big\|_{L_t^2L_x^\frac{2d}{d+2}(I\times\R^d)}\\\label{nlesf3}
&+\Big\||\nabla|^{s_c-2}\Big[\nabla
u\cdot\big(f'(u+v)-f'(u)\big)\Big]\Big\|_{L_t^2L_x^\frac{2d}{d+2}(I\times\R^d)}.
\end{align}
One one hand, by Lemma \ref{moser}, fractional chain rule
\eqref{hfscls}, H\"older inequality and interpolation, we obtain
\begin{align*}
\big\||\nabla|^{s_c-2}\big[\nabla v\cdot
|v|^p\big]\big\|_{L_t^2L_x^\frac{2d}{d+2}(I\times\R^d)}
\lesssim&\big\||\nabla|^{s_c-1}v\big\|_{L_t^{q_2}L_x^{r_2}}\|v\|_{X^0(I)}^p\\
\lesssim&\big\||\nabla|^{s_c}v\big\|_{L_t^{q_2}L_x^{r_3}}\|v\|_{X^0(I)}^p\\
\lesssim&\big\||\nabla|^{s_c}v\big\|_{S^0(I)}\|v\|_{X^0(I)}^p,
\end{align*}
where
\begin{equation*}
\begin{cases}
\frac12=\frac1{q_2}+\frac{p}{q_0},\Rightarrow~q_2=\frac{2q_0}{q_0-p},\\
\frac{d+2}2=\frac{d}{r_2}+\frac{(d-p)p}{p+2}\\
\frac4{q_2}=d\big(\frac12-\frac1{r_3}\big),\Rightarrow~r_3=\frac{2q_0d}{(d-4)q_0+4p},\\
1-\frac{d}{r_3}=-\frac{d}{r_2}.
\end{cases}
\end{equation*}
Hence
\begin{align*}
&\big\||\nabla|^{s_c-2}\big[\nabla v\cdot
f'(u+v)\big]\big\|_{L_t^2L_x^\frac{2d}{d+2}(I\times\R^d)}\\
\lesssim&\big\||\nabla|^{s_c}v\big\|_{S^0(I)}\big(\|u\|_{X^0(I)}^p+\|v\|_{X^0(I)}^p\big)+\big\||\nabla|^{s_c}v\big\|_{S^0(I)}^\frac1{s_c}\|v\|_{X^0(I)}^
{1-\frac1{s_c}}\|u+v\|_{X^0(I)}^{p-1+\frac1{s_c}}\big\||\nabla|^{s_c}(u+v)\big\|_{S^0(I)}^{1-\frac1{s_c}}\\
\lesssim&\big\||\nabla|^{s_c}v\big\|_{S^0(I)}\Big[\|v\|_{X^0(I)}^p+\|u\|_{X^0(I)}^{p-1+\frac1{s_c}}\|v\|_{X^0(I)}^{1-\frac1{s_c}}+\big(
\|u\|_{X^0(I)}^{p-1+\frac1{s_c}}+\|v\|_{X^0(I)}^{p-1+\frac1{s_c}}\big)\big\||\nabla|^{s_c}v\big\|_{S^0(I)}^{1-\frac1{s_c}}\Big].
\end{align*}
On the other hand, by Lemma \ref{yashuo}, H\"older inequality,
interpolation and  \eqref{inter1}, one has
\begin{align*}
&\Big\||\nabla|^{s_c-2}\Big[\nabla
u\cdot\big(f'(u+v)-f'(u)\big)\Big]\Big\|_{L_t^2L_x^\frac{2d}{d+2}(I\times\R^d)}\\
\lesssim&\big\||\nabla|^{s_c}u\big\|_{S^0(I)}\|v\|_{X^0(I)}^p+\big\||\nabla|^{s_c}u\big\|_{S^0(I)}\|u\|_{X^0(I)}^{\frac{s_c-1}\sigma}
\|v\|_{X^0(I)}^{p-\frac{s_c-1}\sigma}\\
&+\big\||\nabla|^{s_c}u\big\|_{S^0(I)}^\frac1{s_c}\|u\|_{X^0(I)}^{1-\frac1{s_c}}
\big\||\nabla|^{s_c}v\big\|_{S^0(I)}^{1-\frac1{s_c}}\|v\|_{X^0(I)}^{p-1+\frac1{s_c}},
\end{align*}
where $s_c-1<\sigma p<p.$ Letting $\beta:=\frac{s_c-1}{\sigma}$, we
derive \eqref{nlesf11} for $p\leq1$.We can iterate the argument
presented above to obtain \eqref{nlesf11} for $p>1$. This concludes
the proof of this lemma.
\end{proof}

Before we prove the perturbation. We first show the short-time
perturbations.

\begin{lemma}[short-time perturbation]\label{shorttimep}
Let $d\geq9$,   $p$ be not an even integer and $1\leq s_c<2+p.$ Let
$I$ be a compact time interval and $u,~\tilde{u}$ satisfy
\begin{align*}
(i\partial_t+\Delta^2)u=&-f(u)+eq(u)\\
(i\partial_t+\Delta^2)\tilde{u}=&-f(\tilde{u})+eq(\tilde{u})
\end{align*}
for some function $eq(u),eq(\tilde{u})$, and $f(u)=|u|^pu$.
 Assume that for some constants $E>0$, we have
\begin{align}\label{eq2.22310}
\|u\|_{L_t^\infty(I;
\dot{H}^{s_c}_x(\R^d))}+\|\tilde{u}\|_{L_t^\infty(I;
\dot{H}^{s_c}_x(\R^d))}\leq E.
\end{align} Moreover, for $t_0\in I$,  and assume that  smallness conditions
\begin{align}\label{stabxx1}
\|\tilde{u}\|_{X(I)}\leq&\delta\\\label{stabxx2}
\big\|u(t_0)-\tilde{u}(t_0)\big\|_{\dot{H}^{s_c}}\leq&\varepsilon\\\label{stabxx3}
\big\||\nabla|^{s_c-1}\big(eq(u),eq(\tilde{u})\big)\big\|_{L_t^2L_x^\frac{2d}{d+2}(I\times\R^d)}\leq&\varepsilon
\end{align}
for some small $0<\varepsilon<\varepsilon_1=\varepsilon_1(M,E)$ and
$0<\varepsilon<\varepsilon_0(E).$ Then  we conclude that
\begin{align}\label{shortte1}
\|u-\tilde{u}\|_{X(I)}\lesssim&\varepsilon\\\label{shortte2}
\big\||\nabla|^{s_c}(u-\tilde{u})\big\|_{S^{0}(I)}\lesssim&
\varepsilon^{c(d,p)}\\\label{shortte3}
\big\||\nabla|^{s_c}u\big\|_{S^{0}(I)}\lesssim& E\\\label{shortte4}
\big\|f(u)-f(\tilde{u})\big\|_{Y(I)}\lesssim&\varepsilon\\\label{shortte5}
\big\||\nabla|^{s_c-1}\big(f(u)-f(\tilde{u})\big)\big\|_{L_t^2L_x^\frac{2d}{d+2}(I\times\R^d)}\lesssim&\varepsilon^{c(d,p)},
\end{align}
for some positive constant $c(d,p)$.
\end{lemma}

\begin{proof}

{\bf Step 1: We claim that
$\big\||\nabla|^{s_c}\tilde{u}\big\|_{S^0(I)}\lesssim E$.}

Indeed, using Strichartz estimate, Corollary \ref{fscqdfz},
 \eqref{stabxx1} and
\eqref{stabxx3}, we get
\begin{align*}
\big\||\nabla|^{s_c}\tilde{u}\big\|_{S^0(I)}\lesssim&\|\tilde{u}\|_{L_t^\infty\dot{H}^{s_c}}+\big\||\nabla|^{s_c}f(\tilde{u})\big\|_{N^0(I)}
+\big\||\nabla|^{s_c}e\big\|_{N^0(I)}\\
\lesssim&E+\|u\|_{L_{t,x}^\frac{(d+4)p}4(I\times\R^d)}^p\big\||\nabla|^{s_c}\tilde{u}\big\|_{S^0(I)}+\varepsilon\\
\lesssim&E+\delta^{p\theta_2}\big\||\nabla|^{s_c}\tilde{u}\big\|_{S^0(I)}^{1+p(1-\theta_2)}+\varepsilon.
\end{align*}
Hence, by the standard bootstrap argument, and choosing
$\delta,~\varepsilon_0$ to be sufficiently small, we obtain
\begin{equation}\label{step1tilde}
\big\||\nabla|^{s_c}\tilde{u}\big\|_{S^0(I)}\lesssim E.
\end{equation}

{\bf Step 2: We claim that $\|u\|_{X(I)}\lesssim\delta.$}

By Lemma \ref{exto}, \eqref{nlesfs1}, \eqref{stabxx1} and
\eqref{stabxx3}, one has
$$\big\|e^{i(t-t_0)\Delta^2}\tilde{u}(t_0)
\big\|_{X(I)}\lesssim\|\tilde{u}\|_{X(I)}+\|f(\tilde{u})\|_{Y(I)}+\big\||\nabla|^{s_c}eq(\tilde{u})\big\|_{N^0(I)}
\lesssim \delta+\delta^{p+1}+\varepsilon\lesssim\delta.$$ Combining
this with the triangle inequality, \eqref{inter1}, Strichartz
estimate and \eqref{stabxx2}, we derive
\begin{align*}
\big\|e^{i(t-t_0)\Delta^2}u(t_0)\big\|_{X(I)}\lesssim&\big\|e^{i(t-t_0)\Delta^2}\tilde{u}(t_0)
\big\|_{X(I)}+\big\|e^{i(t-t_0)\Delta^2}(u-\tilde{u})(t_0)\big\|_{X(I)}\\
\lesssim&\delta+\big\|u(t_0)-\tilde{u}(t_0)\big\|_{\dot{H}^{s_c}}\lesssim\delta.
\end{align*}
On the other hand, by Lemma \ref{exto} and Lemma \ref{nlefs}, we get
$$\|u\|_{X(I)}\lesssim\big\|e^{i(t-t_0)\Delta^2}u(t_0)\big\|_{X(I)}+\|f(u)\|_{Y(I)}+
\big\||\nabla|^{s_c-1}eq(u)\big\|_{L_t^2L_x^\frac{2d}{d+2}(I\times\R^d)}\lesssim\delta+\|u\|_{X(I)}^{p+1}.$$
Thus, we obtain by the bootstrap argument
\begin{equation}\label{uxix}
\|u\|_{X(I)}\lesssim\delta.
\end{equation}

{\bf Step 3:} Next we prove the following iteration formula
\begin{align}\label{iter21}
\|\omega\|_{X(I)}\lesssim\varepsilon+\big\||\nabla|^{s_c}\omega\big\|_{S^0(I)}^\frac{p}{s_c}
\|\omega\|_{X(I)}^{1+\frac{p(s_c-1)}{s_c}},\\\label{iter22}
\big\||\nabla|^{s_c}\omega\big\|_{S^0(I)}\lesssim\varepsilon+\|\omega\|_{X(I)}^{p-\beta}+\big\||\nabla|^{s_c}\omega\big\|_{S^0(I)}^{1-\frac1{s_c}}
\|\omega\|_{X(I)}^{p-1+\frac1{s_c}},
\end{align}
where $\omega=u-\tilde{u}$ satisfies the difference equation
\begin{equation}\label{xianchapneq2}
\left\{\begin{aligned}
&i\omega_t+\Delta^2 \omega=-f(\tilde{u}+\omega)+f(\tilde{u})+eq(u)-eq(\tilde{u}), \\
&\omega(t_0,x)=u(t_0,x)-\tilde{u}(t_0,x)\in \dot H^{s_c}(\R^d).
\end{aligned}\right.
\end{equation}
Using Lemma \ref{exto}, Strichartz estimate, \eqref{stabxx1} and
\eqref{stabxx2}, we get
\begin{align}\nonumber
\|\omega\|_{X(I)}\lesssim&\big\|u(t_0)-\tilde{u}(t_0)\big\|_{\dot{H}^{s_c}}+
\big\||\nabla|^{s_c-1}\big(eq(u),eq(\tilde{u})\big)\big\|_{L_t^2L_x^\frac{2d}{d+2}(I\times\R^d)}+\big\|f(u)-f(\tilde{u})\big\|_{Y(I)}\\\label{stabomw1}
 \lesssim&\varepsilon+\big\|f(u)-f(\tilde{u})\big\|_{Y(I)}.
\end{align}
{\bf The estimate of $\big\|f(u)-f(\tilde{u})\big\|_{Y(I)}$:} From
Lemma \ref{nlefs}, \eqref{stabxx1} and Step 1: \eqref{step1tilde},
we know that
\begin{align}\nonumber
\big\|f(u)-f(\tilde{u})\big\|_{Y(I)}\lesssim&\Big[\|\tilde{u}\|_{X(I)}^\frac{p(s_c-1)}{s_c}\big\||\nabla|^{s_c}\tilde{u}\big\|_{S^0(I)}^\frac{p}{s_c}
+\|\omega\|_{X(I)}^\frac{p(s_c-1)}{s_c}\big\||\nabla|^{s_c}\omega\big\|_{S^0(I)}^\frac{p}{s_c}\Big]\|\omega\|_{X(I)}\\\label{xichaf1}
\lesssim&\delta^\frac{p(s_c-1)}{s_c}E^\frac{p}{s_c}\|\omega\|_{X(I)}+\big\||\nabla|^{s_c}\omega\big\|_{S^0(I)}^\frac{p}{s_c}
\|\omega\|_{X(I)}^{1+\frac{p(s_c-1)}{s_c}}.
\end{align}
Plugging this into \eqref{stabomw1}, and taking $\delta$
sufficiently small, we have
\begin{equation}\label{stabxhdied1}
\|\omega\|_{X(I)}\lesssim\varepsilon+\big\||\nabla|^{s_c}\omega\big\|_{S^0(I)}^\frac{p}{s_c}
\|\omega\|_{X(I)}^{1+\frac{p(s_c-1)}{s_c}}.
\end{equation} This is \eqref{iter21}.

On the other hand, using Strichartz estimate, \eqref{stabxx2} and
\eqref{stabxx3}, we obtain
\begin{align}\nonumber
&\big\||\nabla|^{s_c}\omega\big\|_{S^0(I)}\\\nonumber
\lesssim&\big\|u(t_0)-\tilde{u}(t_0)
\big\|_{\dot{H}^{s_c}}+\big\||\nabla|^{s_c-1}\big(eq(u),eq(\tilde{u})\big)\big\|_{L_t^2L_x^\frac{2d}{d+2}(I\times\R^d)}\\\nonumber&+
\big\||\nabla|^{s_c-1}\big(f(u)-f(\tilde{u})\big)\big\|_{L_t^2L_x^\frac{2d}{d+2}(I\times\R^d)}\\\label{stabxhdd}
\lesssim&\varepsilon+\big\||\nabla|^{s_c-1}\big(f(u)-f(\tilde{u})\big)\big\|_{L_t^2L_x^\frac{2d}{d+2}(I\times\R^d)}.
\end{align}
{\bf The estimate of
$\big\||\nabla|^{s_c-1}\big(f(u)-f(\tilde{u})\big)\big\|_{L_t^2L_x^\frac{2d}{d+2}(I\times\R^d)}$
:} By \eqref{nlesf11},  \eqref{stabxx1}, Step 1: \eqref{step1tilde}
and Step 2: \eqref{uxix}, one has
\begin{align}\label{xichaf2}
&\big\||\nabla|^{s_c-1}\big(f(u)-f(\tilde{u})\big)\big\|_{L_t^2L_x^\frac{2d}{d+2}(I\times\R^d)}\\\nonumber
\lesssim&\big\||\nabla|^{s_c}\omega\big\|_{S^0(I)}\big(\delta+\delta^{p-1+\frac1{s_c}}E
^{1-\frac1{s_c}}\big)
+E\delta^\beta\|\omega\|_{X(I)}^{p-\beta}+\delta^{1-\frac1{s_c}}E^{\frac1{s_c}}\big\||\nabla|^{s_c}\omega\big\|_{S^0(I)}^{1-\frac1{s_c}}
\|\omega\|_{X(I)}^{p-1+\frac1{s_c}}.
\end{align}
Plugging this into \eqref{stabxhdd}, we have
\begin{equation}
\big\||\nabla|^{s_c}\omega\big\|_{S^0(I)}\lesssim\varepsilon+\|\omega\|_{X(I)}^{p-\beta}+\big\||\nabla|^{s_c}\omega\big\|_{S^0(I)}^{1-\frac1{s_c}}
\|\omega\|_{X(I)}^{p-1+\frac1{s_c}}.
\end{equation} This is \eqref{iter22}.
Putting this into \eqref{stabxhdied1}, and by the bootstrap
argument, we conclude  \eqref{shortte1} and \eqref{shortte2}.

And then, it is easy to show \eqref{shortte3} by the triangle
inequality, \eqref{step1tilde} and \eqref{shortte2}.

Using \eqref{shortte1}, \eqref{shortte2}, \eqref{xichaf1} and
\eqref{xichaf1}, we obtain \eqref{shortte4} and \eqref{shortte5}. We
concludes the proof of this lemma.
\end{proof}

Now we turn to prove Lemma \ref{longpert1}.

\noindent{\bf The proof of Lemma \ref{longpert1}:}  First, we claim
\begin{equation}\label{longtsu}
\big\||\nabla|^{s_c}\tilde{u}\big\|_{S^0(I)}\leq C(E,M).
\end{equation}
In fact, from the hypothesis \eqref{eq2.2230}, we know that one can
subdivide time interval  $I$ by $I=\cup_j
I_j,~I_j=[t_j,t_{j+1}],~0\leq j<J_0=J_0(M,\eta)$, such that
$$\|\tilde{u}\|_{L_{t,x}^\frac{p(d+4)}4(I_j\times\R^d)}\leq\eta,$$
where $\eta>0$ is sufficiently small to be determined. Using
Strichartz estimate, fractional chain rule \eqref{fscqdfz},
\eqref{eq2.2230} and \eqref{eq2.22212}, we get
\begin{align*}
\big\||\nabla|^{s_c}\tilde{u}\big\|_{S^0(I_j)}\lesssim&\|\tilde{u}(t_j)\|_{\dot{H}^{s_c}}+
\big\||\nabla|^{s_c-1}eq(\tilde{u})\big\|_{L_t^2L_x^\frac{2d}{d+2}(I_j\times\R^d)}+
\big\||\nabla|^{s_c-1}f(\tilde{u})\big\|_{L_t^2L_x^\frac{2d}{d+2}(I_j\times\R^d)}\\
\lesssim&E+\varepsilon+\|\tilde{u}\|_{L_{t,x}^\frac{p(d+4)}4(I_j\times\R^d)}^p\big\||\nabla|^{s_c}\tilde{u}\big\|_{S^0(I_j)}\\
\lesssim&E+\varepsilon+\eta^p\big\||\nabla|^{s_c}\tilde{u}\big\|_{S^0(I_j)}.
\end{align*}
Thus, by the bootstrap argument, we have
$$\big\||\nabla|^{s_c}\tilde{u}\big\|_{S^0(I_j)}\lesssim
E+\varepsilon.$$ Summing the above bound over all subinteraval
$I_j$, we get the claim \eqref{longtsu}. In particular, from Sobolev
embedding \eqref{inter1}, we know that
\begin{equation}
\|\tilde{u}\|_{X(I)}\leq C(E,M).
\end{equation}
Hence,  we can subdivide time interval  $I$ by $I=\cup_j
I_j,~I_j=[t_j,t_{j+1}],~0\leq j<J_1=J_1(M,\eta)$, such that
$$\|\tilde{u}\|_{X(I_j)}\leq\delta,$$
where  $\delta$ is as in Lemma \ref{shorttimep}.

Therefore, we can apply  Lemma \ref{shorttimep} to each $I_j$ . And
so, $\forall~0\leq j<J_1,~0<\varepsilon<\varepsilon_1$,
\begin{equation}\label{bounds}
\begin{split}
\|u-\tilde{u}\|_{X(I_j)}\lesssim&\varepsilon\\
\big\||\nabla|^{s_c}(u-\tilde{u})\big\|_{S^{0}(I_j)}\lesssim&
\varepsilon^{c(d,p)}\\
\big\||\nabla|^{s_c}u\big\|_{S^{0}(I_j)}\lesssim& E\\
\big\|f(u)-f(\tilde{u})\big\|_{Y(I_j)}\lesssim&\varepsilon\\
\big\||\nabla|^{s_c-1}\big(f(u)-f(\tilde{u})\big)\big\|_{L_t^2L_x^\frac{2d}{d+2}(I_j\times\R^d)}\lesssim&\varepsilon^{c(d,p)},
\end{split}
\end{equation}
provided that one can prove for any $0\leq j<J_1$
\begin{equation}\label{sml}
\|u(t_j)-\tilde{u}(t_j)\|_{\dot{H}_x^{s_c}(I_j)}\leq
C_{j}\varepsilon^{c(d,p)^j}\leq\varepsilon_0.
\end{equation}
Indeed, by Strichartz estimate and the inductive hypothesis, one has
\begin{align*}
&\|u(t_j)-\tilde{u}(t_j)\|_{\dot{H}_x^{s_c}(I_j)}\\
\lesssim&\|u_0-\tilde{u}_0\|_{\dot{H}_x^{s_c}(I_j)}+
\big\||\nabla|^{s_c-1}\big(eq(u),eq(\tilde{u})\big)\big\|_{L_t^2L_x^\frac{2d}{d+2}([0,t_j]\times\R^d)}
\\
&+\big\||\nabla|^{{s_c-1}}\big[f(\tilde{u}+\omega)-f(\tilde{u})\big]\big\|_{L_t^2L_x^\frac{2d}{d+2}([0,t_j]\times\R^d)}\\
\lesssim&\varepsilon+\sum_{k=0}^{j-1}C_k\varepsilon^{c(d,p)^k}.
\end{align*}
Taking $\varepsilon_1$ sufficiently small compared to
$\varepsilon_0$, we derive \eqref{sml}.

Summing the bounds in \eqref{bounds} over all subintervals $I_j$, we
conclude Lemma \ref{longpert1}.

\textbf{Acknowledgements} The authors  were  supported by the NSF of China under
grant No.11171033, 11231006.

\begin{center}

\end{center}

\end{document}